\documentclass{amsart}
\usepackage{amsmath}
\usepackage{amssymb}
\usepackage{amsthm}
\usepackage{amscd}
\usepackage[all]{xy}
\usepackage{color,mathrsfs, mathtools}
\usepackage{comment}

\newcommand{\annette}[1]{{\color{red}Annette:#1}}
\newcommand{\guido}[1]{{\color{red}Guido:#1}}
\newcommand{\DM}{\mathbf{DM}}
\newcommand{\DA}{\mathbf{DA}}
\newcommand{\MT}{\mathbf{MT}}
\newcommand{\Log}[1]{\mathcal{L}og^{(#1)}}
\newcommand{\pol}[1]{\mathrm{pol}^{(#1)}}
\newcommand{\polzwei}[1]{\overline{\cpol}^{(#1)}}
\newcommand{\cpol}{\mathrm{pol}}
\newcommand{\cpolzwei}{\overline{\mathrm{pol}}}
\newcommand{\cLog}{\mathcal{L}og}
\newcommand{\Sym}{\mathrm{Sym}}
\newcommand{\Hom}{\mathrm{Hom}}
\newcommand{\id}{\mathrm{id}}
\newcommand{\kd}{\mathrm{kd}}
\newcommand{\tensor}{\otimes}
\newcommand{\Ext}{\mathrm{Ext}}
\newcommand{\isom}{\cong}
\newcommand{\gr}{\operatorname{gr}}
\newcommand{\comp}{\operatorname{comp}}
\newcommand{\ev}{\operatorname{ev}}

\newcommand{\coker}{\operatorname{coker}}
\newcommand{\et}{\mathrm{et}}
\newcommand{\Sm}{\mathrm{Sm}}
\newcommand{\Sh}{\mathrm{Sh}}
\newcommand{\Gm}{\mathbb{G}_m}
\newcommand{\Spec}{\mathrm{Spec}}
\newcommand{\one}[1]{\mathbf{1}^{(#1)}}
\newcommand{\res}{\mathrm{res}}

\newcommand{\ul}[1]{\underline{#1}}
\newcommand{\Kh}{\mathcal{K}}
\newcommand{\Ah}{\mathcal{A}}

\newcommand{\Hh}{\sH}
\newcommand{\Th}{\mathcal{T}}
\newcommand{\Q}{\mathbb{Q}}
\newcommand{\Ql}{\mathbb{Q}_\ell}
\newcommand{\Z}{\mathbb{Z}}
\newcommand{\R}{\mathbb{R}}
\newcommand{\C}{\mathbb{C}}
\newcommand{\D}{\mathbf{D}}
\newcommand{\sF}{\mathscr{F}}
\newcommand{\sG}{\mathscr{G}}
\newcommand{\sH}{\mathscr{H}}

\newcommand{\bQ}{\mathbf{Q}}

\newcommand{\an}{{\mathrm{an}}}
\renewcommand{\phi}{\varphi}
\newcommand{\cL}{\mathcal{L}}

\theoremstyle{definition}
\newtheorem{defn}{Definition}[subsection]
\newtheorem{ex}[defn]{Example}
\newtheorem{rem}[defn]{Remark}

\theoremstyle{plain}
\newtheorem{lemma}[defn]{Lemma}
\newtheorem{prop}[defn]{Proposition}
\newtheorem{theorem}[defn]{Theorem}
\newtheorem{cor}[defn]{Corollary}

\begin{document}
\title{Polylogarithm for families of commutative group schemes}
\author{Annette Huber}\address[Huber]{Mathematisches Institut \\
Albert-Ludwigs-Universit\"at Freiburg \\ 
79104 Freiburg }
\email{annette.huber@math.uni-freiburg.de}
\author{Guido Kings}\address[Kings]{Fakult\"at f\"ur Mathematik \\
Universit\"at Regensburg\\
93040 Regensburg}
\email{guido.kings@mathematik.uni-regensburg.de}
\thanks{The research of G. Kings was supported by the DFG through SFB 1085: Higher invariants}
\begin{abstract}We generalize the definition of the polylogarithm classes
to the case of commutative group schemes, both in the sheaf theoretic
and the motivic setting. This generalizes and simplifies the existing
cases.
\end{abstract}
\maketitle
\tableofcontents
\section{Introduction}
Since its invention by Deligne, the importance of the cyclotomic polylogarithm and its elliptic analogue increased with each new aspect discovered about it. The main reason for this is the fact that the polylogarithm remains the only systematic way to construct interesting classes
in motivic cohomology and that its realizations are related to important functions like Euler's polylogarithm or real analytic Eisenstein series. Many results about special values of $L$-functions rely on the motivic classes of the polylogarithm and we just mention
the Tamagawa number conjecture for abelian number fields (\cite{HKDuke} and \cite{BurnsGreither}), for CM elliptic curves (\cite{Ki-Invent}) and modular forms (\cite{Gealy}), or Kato's work on the conjecture of Birch and Swinnerton-Dyer (\cite{Kato}).

It was already a vision of Beilinson and Levin (unpublished) that it should be possible to define  the polylogarithm for general $K(\pi,1)$-spaces, a program realized to a large extent by Wildeshaus in \cite{Wi}. There the polylogarithm was defined for extensions of abelian schemes by tori, a restriction which is unfortunate when dealing with degenerations, and the motivic construction of the polylogarithm was lacking.

In this paper we propose a new definition of the polylogarithm which works for arbitrary smooth commutative group schemes with connected fibres. This is not quite a generalization of Wildeshaus' definition (it agrees with it in some special cases, e.g. for abelian schemes), but the better functoriality properties of our definition make this look like the right construction. What is more, and highly important for  applications, we can construct a class in motivic cohomology for our polylogarithm building on the techniques and results developed in \cite{AHP} and \cite{Ki}.\\

To explain the novel features in our construction, let us briefly review the definition of the polylogarithm (as we propose it) in the sheaf theoretic setting.
Let $S$ be noetherian finite dimensional scheme, $\pi:G\to S$  a smooth commutative group scheme with connected fibres of dimension $d$.
Let 
\[
\sH:=\sH_G:=R^{2d-1}\pi_!\bQ(d)=R^{-1}\pi_!\pi^!\bQ
\]
be the first homology of the group scheme. This is the sheaf of the Tate-modules
of the fibres. The main player is the universal Kummer extension 
\[
0\to \pi^*\sH\to \Log{1}\to \bQ\to 0
\]
on $G$. Taking symmetric powers $\Log{n}:=\Sym^{n}\Log{1}$ one gets a projective system of sheaves $\cLog$. The $\Log{n}$ have obviously a filtration whose associated graded are
just the $\Sym^{n}\sH$. Moreover, $\cLog$  has the important property that for torsion sections $t:S\to G$ one has 
\[ t^*\cLog=\prod_{n=0}^\infty\Sym^n\Hh\]
(as pro-objects) which is called the splitting principle. This applies
in particular to the unit section $e$.
%In particular, for the unit section $e:S\to G$ one has 
%\[
%\prod_{n=1}^{\infty}\Sym^{n}\sH=\ker(e^{*}\cLog\to \bQ).
%\]
The pro-object $\cLog$ together with the splitting is characterized by a universal
property, which we are able to 
verify in the sheaf theoretic setting (Theorem \ref{thm:sheaf-univ-property}) and under some
more restrictive assumptions also in the motivic setting (see Theorem \ref{thm:univ-property}).

We then turn to the construction of the polylogarithm.
Let $j:U:=G\setminus e(S)\to {G}$ be the open immersion of the complement of the unit section.
The polylogarithm is a class 
\[\cpol\in \Ext^{2d-1}_{S}(\sH,R\pi_!Rj_*j^{*}\cLog(d)) \]
whose image under the residue map 
\[ \Ext^{2d-1}_{S}(\sH,R\pi_!Rj_*j^{*}\cLog(d))\xrightarrow{\res}
\Hom_S(\sH, e^*\cLog)
\]
is given by the natural inclusion $\Hh\to \prod_{n=0}^{\infty}\Sym^{n}\sH$. 
The difference of our definition to the existing ones in the literature is the use of  $R\pi_!$. In fact it is one of our main insights
that everything becomes much more natural using cohomology with compact support.

In the sheaf theoretic setting 
the existence and uniqueness of $\cpol$ follows from the
vanishing of the higher direct images of $R\pi_!\cLog$. In  the motivic setting, we cannot make the same computation. However, analyzing
the operation of multiplication by $a\in\Z$ we get a decomposition
of $R\pi_!\cLog$ into generalized eigenspaces. We get existence and a unique characterization of $\cpol$
when asking it in addition to be in the right eigenspace. 
By either approach, the classes can easily be seen to be natural with respect to both $S$ and $G$. By construction, the realization functors map
the motivic classes to the sheaf theoretic ones.

We would also like to advocate a slight variant of the above definition, which appears already in \cite{BeLe} but not so much in other literature on the polylog. For each $\Q$-valued function $\alpha$ of degree $0$ on a finite subscheme $D$ of torsion points one can define
\[
\cpol_\alpha\in\Ext^{2d-1}(\bQ,R\pi_!j_{D*}j_D^*\cLog(d)).
\]
This class has the advantage of having very good norm compatibility properties, which are useful in Iwasawa theoretic applications
(see \cite{Ki-Eisenstein}).\\

How can we have a more general motivic construction and still a simpler one? The
main reason is that by the work of Ayoub and Cisinski-Deglise the theory of triangulated motives over a general base has now been developed to a point
that makes calculations possible. One such is the computation
of motives of commutative groups schemes in  \cite{AEH}. The original constructions could only use motivic cohomology with coefficients in $\bQ(j)$. All the interesting non-constant nature of $\cLog$ had to
be encoded in complicated geometric objects. In the case of the classical
polylog, the basic object $\Log{1}$ had to be defined using relative
cohomology - forcing the use of simplicial schemes in \cite{HuWi}.
We are still missing the motivic t-structure on triangulated motives, but
in our case $[a]$-eigenspace arguments as in \cite{Ki}, which generalize \cite{BeLe},
 can be used as a replacement. Indeed, also this
part of the argument is clarified by applying it to objects rather than
cohomology groups. For a complete list of earlier results, see the discussion in
Section~\ref{sec:special_cases}.

What is missing in contrast to the cases already in the literature is
an explicit description of the monodromy matrices of $\cLog$ and $\cpol$ or the computation of other realizations.

\subsection*{Organization of the paper}
The paper starts with a section on notation; fixing the geometric situation
and also explaining the various settings we are going to work in.

Section \ref{sec:sheaflog} gives the sheaf theoretic construction
of $\cLog$, including the  formulation of the universal property.
Section \ref{sec:motlog} mimicks the construction in the motivic
setting.

From this point on, we work in parallel in the sheaf theoretic and motivic
setting. Section \ref{sec:polylog} explains the polylogarithm
extension and its properties. In Section \ref{sec:with_*} we relate
the present construction to the ones in the literature. The particularly
important case of the cyclotomic polylog is discussed in more detail.
Finally, Section \ref{proof:vanishing} provides a couple of longer, technical
proofs on properties of $\cLog$, which had been delayed for reasons 
of readability.

An appendix discusses the decomposition into
generalized eigenspaces in general $\Q$-linear triangulated categories.

 \subsection*{Acknowledgements} It should be already clear from the introduction how much we are influenced by the ideas and constructions of Beilinson-Levin and Deligne-Beilinson. It is pleasure to thank F. Ivorra and S. Pepin-Lehalleur
for discussions.

\section{Setting and preliminaries}\label{sec:setting}

\subsection{Geometric situation}\label{sec:setting_geometry}
We fix the following notation. Let $S$ be a base scheme, subject to further conditions depending on the setting. Let
\[ \pi:G\to S\]
be a smooth commutative group scheme with connected fibres of relative dimension $d$ and unit section $e:S\to G$ and multiplication $ \mu:G\times_SG\to G$.
Let $j:U\to G$ be the open complement of $e(S)$.

Let $\iota_D:D\to G$ be a closed subscheme with structural map
$\pi_D:D\to S$. Most of the time we will assume $\pi_D$ \'etale and
$D$ contained in the $N$-torsion of $G$ for some $N\geq 1$.  
Let 
$j_D:U_D=G\setminus D\to G$ be the open complement of $D$. This basic set up is summarized in the diagram
\[
\xymatrix{
U_D:=G\setminus D\ar[r]^/.6em/{j_D}\ar[rd]&G\ar[d]^{\pi}&D\ar[l]_{\iota_D}\ar[dl]^{\pi_D}\\
&S
}
\]
We will also consider morphisms $\phi:G_1\to G_2$ of $S$-group schemes as above. In this case we decorate all notation with an index $1$ or $2$, e.g., $d_1$
for the relative dimension of $G_1/S$.

\subsection{$\ell$-adic setting}\label{sec:setting_ladic}
Let $S$ be of finite type over a regular scheme of dimension $0$ or $1$.
Let $\ell$ be a prime invertible on $S$,
$X\to S$ separated and of finite type.
We work in the category of constructible $\Ql$-sheaves on $X$ in the
sense of \cite[Expos\'e V]{sga5} and its ``derived'' category
in the sense of Ekedahl \cite{Eke}. They are triangulated categories with a $t$-structure whose heart is the category of constructible $\Ql$-sheaves.
By loc. cit. Theorem 6.3 there is a full 6 functor formalism on these categories. 
%Moreover, the triangulated categories allow a $t$-structure whose heart is given by constructible $\Ql$-sheaves.

\subsection{Analytic sheaves}\label{sec:setting_analytic}
Let $S$ be separated and of finite type over the complex numbers. For $X\to S$ separated and of finite type, we denote $X^\an$ the set $X(\C)$ equipped with the analytic topology. We work in the category of constructible sheaves of $\Q$-vector spaces on $X^\an$ and its derived category. There 
is a full 6 functor formalism on these categories, see e.g. \cite{di04}.

\subsection{Hodge theoretic setting}\label{sec:setting_hodge}
Let again $S$ be separated and of finite type over the complex numbers.
Let $X\to S$ be separated and of finite type. We work in the derived category of
Hodge modules on $X$ of Saito, e.g. \cite{sai88}. It has a natural forgetful functor into the derived category of constructible sheaves on $X^\an$.
By \cite[Section 4.6 Remarks 2. page 328-329]{sai90}  it also carries a $t$-structure whose heart maps to the abelian category of constructible sheaves via the forgetful functor. 
Note that this {\em not} the better known $t$-structure whose heart maps to
perverse sheaves.

\subsection{Motivic setting}\label{sec:setting_motivic}
Let $S$ be noetherian and finite dimensional. Let $X\to S$ be separated and of finite type. 

We denote $\DA(S)$ the triangulated category of \'etale motives
without transfers with {\em rational coefficients}. 

This is the same notation as in \cite{AHP}, our main reference in the sequel.
The category is denoted $\DA^\et(S,\Q)$ in the work of Ayoub \cite{Ayou7a}, \cite{Ayou7b}, \cite{ayoubet}. In the work of Cisinski and D\'eglise (see \cite[16.2.17]{CD}) it
is the category $D_{\mathbb{A}^1,\et} (\Sm/S,\Q)$. 

There is a full 6 functor formalism for these categories. In particular, for 
$f:X\to S$ smooth of fibre dimension $d$, there
is a natural object
$M_S(X)\in \DA(S)$. In formulas:
\[ M_S(X)=f_{\#}\Q_X=Rf_!\Q_X(d)[2d]=Rf_!f^!\Q_S.\]

Beside the formal properties of $\DA(S)$, we also are going to use
the existence of a convenient {\em abelian} category mapping to it.
Let $\Sh_\et(\Sm)$ be the category \'etale sheaves of $\Q$-vector spaces on the
category of smooth $S$-schemes of finite type.
Then there is a tensor functor
\[ C^b(\Sh_\et(\Sm))\to \DA(S)\]
which maps short exact sequences to exact triangles.
\begin{rem}There are a number of different triangulated categories of motives
over $S$. With integral or torsion coefficients, 
the differences between them are subtle; and comparison results like the Bloch-Kato conjecture are the
deepest results in the theory. However, the situation is much more straightforward with rational coefficients. For example, we get the same categories
when working with the Nisnevich or the \'etale topology. Under weak 
assumptions on $S$ (e.g., $S$ excellent and regular is more then enough) 
all definitions agree. In these cases, $\DA(S)$ is equivalent to the categories of motives for the qfh-topoly or for the h-topology, to triangulated motives with transfers, and to the category of Beilinson motives of Cisinski and D\'eglise 
\end{rem}

\subsection{Realizations}\label{setting_realizations}
Let $\DA_c(S)$ be the full subcategory of compact objects. 
If $\ell$ is invertible on $S$, then by \cite[Section 9]{ayoubet} there is a covariant
\'etale realization functor
\[ R_\ell:\DA_c(S)\to D_c(S,\Ql)\]
where $\DA_c(S)$ is the full subcategory of compact motives and $D_c(S,\Q_{\ell})$ is the triangulated category of the $\ell$-adic setting. The functors $R_{\ell}$ are compatible with the six functor formalism on both sides and map the Tate
motive $\Q(j)$ to $\Q_{\ell}(j)$.

If $S$ is of finite type over $\C$, then by \cite{Ay10} there is
a covariant Betti realization functor
\[ R_B:\DA_c(S)\to D_c(S^\an,\Q).\]
It is compatible with the six functor formalism on both sides and maps the
Tate motive $\Q(j)$ to $\Q$.

At the time of writing this paper, the situation for the Hodge theoretic
realization is not yet as satisfactory. By work of Drew (\cite{drew-thesis}, \cite{drew-oberwolfach}) there
is realization compatible with the 6 functor formalism into categories
which are of Hodge theoretic flavour but a priori bigger than the
derived category of Hodge modules. By work of Ivorra \cite{Ivorra_Hodge}, 
there is 
realization into Hodge modules for compact motives over a smooth base of finite type over $\C$, but without knowledge about the 6 functors.

\subsection{Notation}
The bulk of our computations will be valid in the various settings without
any changes.
We are going to refer to the $\ell$-adic, analytic or Hodge theoretic setting
by the shorthand {\em sheaf theoretic setting}. By {\em triangulated setting}
we are going to refer to computations on the level of derived categories 
in the $\ell$-adic, analytic or Hodge theoretic setting as well as in the
motivic setting. We denote them uniformly by $\D(X)$.

In any of the above sheaf theories we denote by $\bQ$ the structure
sheaf, i.e., $\Q_\ell$, $\R(0)$. In the motivic setting
we denote $\bQ$ the motive of $S$. It is defined by the image
of the constant \'etale sheaf $\Q$.

To avoid confusion, we write $Rf_*, Rf_!$ etc. for the triangulated functors instead of $f_*$ or $f_!$, which is sometimes used, in particular in \cite{AHP}.
The notation $f_*$, $f_!$ etc. is reserved for the functors between abelian categories of sheaves.

\subsection{Unipotent sheaves}\label{sec:unipotent}
Let $S$ be the base scheme and $\pi:X\to S$ separated and of finite type.

Recall that a sheaf $\sF$ on $X$ is \emph{unipotent of length $n$}, if it has a filtration
$0=\sF^{n+1}\subset \sF^n\subset \ldots\subset \sF^0=\sF$ such that 
$\sF^i/\sF^{i+1}\isom\pi^*\sG^i$ for a sheaf $\sG^i$ on $S$.

In any of the triangulated settings above, we call an object
$M\in \D(X)$ {\em unipotent}
if there is a finite sequence of objects $M_1\to M_2\to\dots M_n=M$
and exact triangles
\[ M_{i-1}\to M_{i}\to \pi_2^*N_{i}.\]
\begin{lemma}\label{lem:unipotent_uppershriek}
Let $\pi_1:X_1\to S$ and $\pi_2:X_2\to S$ be smooth of constant fibre dimension $d_1$ and $d_2$. Let $f:X_1\to X_2$ be an $S$-morphism.
Let $M\in D(X_2)$ be unipotent. Then
\[ f^!M=f^*M(d_1-d_2)[2d_1-2d_2].\]
\end{lemma}
\begin{proof} Put $c=d_1-d_2$ the relative dimension of $f$.
We start with the case $M=\pi_2^*N$. 
In this case
\begin{multline*}
f^!M=f^!\pi_2^*N=f^!\pi_2^!N(-d_2)[-2d_2]=\pi_1^!N(-d_2)[-2d_2]\\
=\pi_1^*N(c)[2c]=f^*\pi_2^*N(c)[2c]
=f^*M\tensor\bQ(c)[2c].
\end{multline*}
In particular, $f^!\bQ=\bQ(c)[2c]$ and we may rewrite the 
 formula as
\[ f^*M\tensor  f^!\bQ=f^!(M\tensor \bQ).\]
There is always a map from the left to right via adjunction from
the projection formula
\[ Rf_!(f^*M\tensor f^!\bQ)= M\tensor Rf_!f^!\Q\to M\tensor \bQ.\]
Hence we can argue on the unipotent length of $M$ and it suffices
to consider the case $M=\pi^*N$. This case was settled above. \end{proof}

Let $X\to S$ be a smooth scheme with connected fibres and $e:S\to X$ a section. Homomorphisms of unipotent sheaves are completely determined by their restriction to $S$ via $e^{*}$:
\begin{lemma} \label{lemma:e-upper-star}
We work in the sheaf theoretic setting.
Let $\pi:X\to S$ be smooth with connected fibres and $e:S\to X$ a section of $\pi$ and $\sF$ a unipotent sheaf on $X$. Then
\[
e^{*}:\Hom_X(\bQ,\sF)\to \Hom_S(e^{*}\bQ,e^{*}\sF)
\]
is injective. 
\end{lemma}
\begin{proof}
Let $0\to\sF_1\to\sF_2\to\sF_3\to 0$ be a short exact sequence of unipotent sheaves
on $X$. By exactness of $e^*$ and left-exactness of $\Hom$ we get a commutative
diagram of exact sequences
\[\xymatrix{%
0\ar[r]&\Hom_X(\bQ,\sF_1)\ar[r]\ar[d]&\Hom_X(\bQ,\sF_2)\ar[r]\ar[d]&\Hom_X(\bQ,\sF_3)\ar[d]\\
0\ar[r]&\Hom_S(\bQ,e^*\sF_1)\ar[r]&\Hom_S(\bQ,e^*\sF_2)\ar[r]&\Hom_S(\bQ,e^*\sF_3)
}\]
If injectivity holds for $\sF_1$ and $\sF_3$, then by a small diagram chase
it also holds for $\sF_2$. Hence by induction on the unipotent length it suffices
 to consider the case $\sF=\pi^*\sG$. We claim that we even have an isomorphism
 in this case. 
 It reads
 \[ \Hom_X(\pi^*\bQ,\pi^*\sG)
\to \Hom_S(\bQ,e^*\pi^*\sG)=\Hom_S(\bQ,\sG).\]
 As $\pi$ is smooth, the left hand side is 
 \[ \Hom_X(\pi^!\bQ,\pi^!\sG)=\Hom_S(R\pi_!\pi^!\bQ,\sG).\]
 Recall that $H^0R\pi_!\pi^!\bQ$ is fibrewise $0$-th homology of $X$. As
 we assume that $\pi$ has connected fibres, this is isomorphic to $\bQ$. 
 Hence
 \[ \Hom_S(H^0R\pi_!\pi^!\bQ,\sG)=\Hom_S(\bQ,\sG).\]
This proves the claim.
\end{proof}
\section{The logarithm sheaf}\label{sec:sheaflog}
We work in one of the sheaf theoretic settings described in Section \ref{sec:setting} and in the geometric situation described there. In particular,
$\pi:G\to S$ is a smooth commutative group scheme with connected
fibres of dimension $d$.

\subsection{Definition of the logarithm sheaf}\label{sec:log}
\begin{defn}\label{defn:tate}
For the group scheme $\pi:G\to S$ we let 
\[
\sH:=\sH_G:=R^{2d-1}\pi_!\bQ(d)=R^{-1}\pi_!\pi^!\bQ.
\]
\end{defn}
The formation of $\sH_G$ is covariant functorial for $S$-group homomorphisms
$\varphi:G_1\to G_2$. The adjunction $\varphi_!\varphi^{!}\bQ\to \bQ$ induces by applying $R\pi_{2!}$ a map of sheaves
\begin{equation}
\varphi_!:\sH_{G_1}\to \sH_{G_2}.
\end{equation}
Using the ''Leray spectral sequence'' for $R\pi_!\pi^!\bQ$ (i.e., the spectral
sequence for the canonical filtration) we get
\begin{multline*}
0\to \Ext^1_S(\bQ,\sH)\xrightarrow{\pi^!}
\Ext^1_G(\pi^!\bQ,\pi^!\sH)\to 
\Hom_S(\sH,\sH)\to\\
\to
\Ext^2_S(\bQ,\sH)\xrightarrow{\pi^!} \Ext^2_G(\pi^!\bQ,\pi^!\sH)
\end{multline*}
and the maps $\pi^{!}$ are injective because they admit
the splitting $e^{!}$ induced by the unit section $e$. This gives
\begin{equation}\label{eq:Log-1}
0\to \Ext^1_S(\bQ,\sH)\xrightarrow{\pi^!}
\Ext^1_G(\pi^!\bQ,\pi^!\sH)\to 
\Hom_S(\sH,\sH)\to 0.
\end{equation}
Note that $\Ext^1_G(\pi^!\bQ,\pi^!\sH)\isom \Ext^1_G(\bQ,\pi^*\sH)$.
\begin{defn}\label{defn:log}
The \emph{first logarithm sheaf} $(\Log{1},\mathbf{1}^{(1)})$ on $G$ consists of an extension class
\[
0\to \pi^*\sH\to \Log{1}\to \bQ\to 0
\]
such that its image in $\Hom_S(\sH,\sH)$ is the identity
together with
a fixed splitting $\mathbf{1}^{(1)}:e^*\bQ\to e^*\Log{1}$. 

We define
\[
\Log{n}:=\Sym^n\Log{1}
\]
and denote by $\mathbf{1}^{(n)}$ the induced splitting $\Sym^n(\mathbf{1}^{(1)}):\bQ\to \Log{n}$.
\end{defn}

The existence and uniqueness of $(\Log{1}_G,\mathbf{1}^{(1)})$ follow directly from \eqref{eq:Log-1}. The automorphisms of $\Log{1}$ form a torsor under $\Hom_G(\bQ,\pi^{*}\sH)$. In particular, the 
pair $(\Log{1},\mathbf{1}^{(1)})$ admits no automorphisms except the
identity.

Consider $\Log{1}\to \Log{1}\oplus \bQ$ induced by the identity and the natural
projection $\Log{1}\to \bQ$. We define transition maps 
\[
\Log{n+1}\isom \Sym^{n+1}\Log{1}\to \Sym^{n+1}(\Log{1}\oplus \bQ)\to \Sym^n\Log{1}\otimes \Sym^1\bQ\isom \Log{n},
\]
induced by the canonical projection. Under these transition maps 
$\mathbf{1}^{(n+1)}$ is mapped to $\mathbf{1}^{(n)}$ and one has an exact sequence
\[
0\to \pi^*\Sym^n\sH\to \Log{n}\to\Log{n-1}\to 0.
\]
This implies that the sheaf $\Log{n}$ is unipotent of length $n$ with associated 
graded $\bigoplus_{k=0}^n\pi^{*}\Sym^k\sH$. The section $\mathbf{1}^{(n)}$
induces an isomorphism
\begin{equation}
e^*\Log{n}\isom \prod_{k= 0}^n\Sym^k\sH.
\end{equation}
\begin{defn}\label{defn:logn}
The \emph{logarithm sheaf} $(\cLog,\mathbf{1})$ is the pro-system of $(\Log{n},\mathbf{1}^{(n)})$
with the above transition maps. The unipotent filtration is given by the kernels
of the augmentation maps
\[
\cLog\to \Log{n}.
\]
\end{defn}

For later reference, we also explain an explicit construction of $\Log{1}$
as the universal Kummer extension.
It is this point of view that will be used in the motivic case.

Note that the unit section induces an isomorphism 
$\bQ\to R^0\pi_!\pi^!\bQ$ and a
splitting 
\begin{equation}\label{eq:coh-splitting} R\pi_!\pi^!\bQ\isom \bQ\oplus \tau_{\leq -1}R\pi_!\pi^!\bQ.
\end{equation}
We apply this to the $G$-group scheme $\tilde{G}=G\times G$ with structure map
$\tilde{\pi}=\pi\times \id$. Its unit section is $\tilde{e}=e\times\id$. The
diagonal $\Delta:G\to G\times G$ is a morphism of $G$-schemes, hence $\Delta$ induces
a natural morphism of functors
\[ \id=R\tilde{\pi}_!R\Delta_!\Delta^!\tilde{\pi}^!\to R\tilde{\pi}_!\tilde{\pi}^!,\]
which we apply to $\bQ$. Together this yields a natural map in $\D(G)$
\begin{equation}\label{eq:comp}
\bQ\to R\tilde{\pi}_!\tilde{\pi}^!\bQ\to \tau_{\leq -1}R\tilde{\pi}_!\tilde{\pi}^!\bQ\to
R^{-1}\tilde{\pi_!}\tilde{\pi}^!\bQ[1]=\pi^*\sH[1].
\end{equation}

\begin{lemma}\label{lem:compute_once}The above composition \eqref{eq:comp} of morphisms in $\D(G)$
agrees with $\Log{1}$ as element of $\Ext^1_G(\bQ,\pi^*\sH)$.
\end{lemma}
\begin{proof}
Let $\cL$ be extension class in the Lemma. 
By Definition \ref{defn:log} we have to check that
\begin{enumerate}
\item $e^*(\cL)=0$ (the $1$-extension is split),
\item the image of $\cL$ in $\Hom_{S}(\Hh, \Hh)$ under the map induced from the Leray spectral sequence 
is
the identity map $\Hh\to \Hh$.
\end{enumerate}
The first statement is true by construction because the restriction
of $\Delta$ and $\tilde{e}$ to the unit section is the unit section $e$.
The splitting of $e^*\cL$ is the one induced from $\tilde{e}$.

We turn to the second statement and review the construction
of the map to $\Hom_S(\Hh,\Hh)$. 
We view $[\cL]$ in $\Hom_G(\pi^!\bQ,\pi^!\sH[1])$. Using
the adjunction between $\pi^!$ and $R\pi_!$ amounts to the composition
\[ R\pi_!\pi^!\bQ\xrightarrow{R\pi_!\cL}R\pi_!\pi^!\sH[1]\to \sH[1].\]
The map ``given by the Leray spectral sequence'' is the one
obtained by precomposing with
\[ \tau_{\leq -1}R\pi_!\pi^!\bQ\to R\pi_!\pi^!\bQ .\]
The result naturally factors via
\[ \sH[1]\to \sH[1]\]
for degree reasons.
The map $R\pi_!\cL$ is induced from 
\[ R\pi_!\bQ\xrightarrow{\Delta_!}R(\pi\times\pi)_!\bQ=
    R\pi_!\bQ\otimes R\pi_!\bQ\to R\pi_!\bQ\otimes \sH[1].\]
We compose with $R\pi_!\pi^!\bQ\to\bQ$ in the first factor.
This agrees with projection to the second factor of $G\times G$,
i.e., to the map induced by the identity.
\end{proof}

\subsection{Functoriality and splitting principle}
We collect some fundamental properties of the logarithm sheaf.

The first important property is the functoriality. Let 
\[
\varphi:G_1\to G_2
\] 
be a homomorphism of group schemes of relative 
dimension $d_1$, $d_2$, respectively, and
$\varphi_!:\sH_{G_1}\to \sH_{G_2}$ be the associated morphism of the homology. 
\begin{theorem}[Functoriality] \label{thm:functoriality} 
Let $c:=d_1-d_2$ be the relative dimension of the homomorphism $\varphi:G_1\to G_2$. Then there is a unique
homomorphism of sheaves
\[
\varphi_\#:\cLog_{G_1}\to \varphi^{*}\cLog_{G_2}\isom \varphi^!\cLog_{G_"}(-c)[-2c]
\]
such that $\mathbf{1}_{G_1}$ maps to $\mathbf{1}_{G_2}$ and which respects the canonical filtrations
on both sides. 
The induced map on the associated graded
 
\[ \gr\varphi_\#:\bigoplus_{n\ge 0}\pi_1^*\Sym^n\sH_{G_1}\to \bigoplus_{n\ge 0}\pi_1^*\Sym^n\sH_{G_2}
\]
coincides with $\Sym^\cdot\varphi_!$. 
If $\varphi$ is an isogeny one has $\varphi_\#:\cLog_{G_1}\to \varphi^!\cLog_{G_2}$.
\end{theorem}
\begin{proof}
We are going to define a homomorphism
\[ \Log{n}_{G_1}\to\phi^*\Log{n}_{G_2}.\]
Assuming this, the right hand side agrees with 
 $\phi^!\Log{n}_{G_2}(-c)[-2c]$ where 
$c=d_1-d_2$ by Lemma \ref{lem:unipotent_uppershriek}. 

As $\phi^*$ is compatible with tensor products,
it suffices to prove the statement for $\Log{1}$. 
The sheaf 
$\varphi^{*}\Log{1}_{G_2}$ defines an extension class in $\Ext^1_{G_1}(\bQ,\pi_1^*\sH_{G_2})$. The push-out of 
$\Log{1}_{G_1}$ by $\varphi_!:\sH_{G_1}\to \sH_{G_2}$ defines also a class in this $\Ext$-group and from the definition one sees that these classes agree. Hence, one has a map of extensions
\[
\begin{CD}
0@>>>\pi_1^{*}\sH_{G_1}@>>>\Log{1}_{G_1}@>>>\bQ@>>> 0\\
@.@V\varphi_!VV@VhVV@|\\
0@>>>\pi_1^{*}\sH_{G_2}@>>>\varphi^{*}\Log{1}_{G_2}@>>>\bQ@>>> 0.
\end{CD}
\]
Taking the pull-back by $e_1^{*}$ and using purity  one gets a splitting
\[
e_1^{*}(h)\circ \mathbf{1}_{G_1}^{(1)}:\bQ\to e_1^{*}\Log{1}_{G_1}\to e_2^{*}\Log{1}_{G_2}.
\]
By uniqueness there is a unique isomorphism of the pair
$(\Log{1}_{G_2},e_1^{*}(h)\circ \mathbf{1}_{G_1}^{(1)})$ with 
$(\Log{1}_{G_2},\mathbf{1}_{G_2}^{(1)})$. The composition of this with
$h$ gives the desired map. 

The difference of any
two maps $h,h':\Log{1}_{G_1}\to \varphi^{*}\Log{1}_{G_2}$ induces 
 a homomorphism $h-h':\bQ\to\pi_1^*\sH$,
 which by Lemma
\ref{lemma:e-upper-star} is uniquely determined by its pull-back 
$e_1^{*}(h-h'):\bQ\to e_2^{*}\Log{1}_{G_2}$. If $h$ and $h'$ are compatible with the splittings the map $e_1^{*}(h-h')$ has to be zero, so that $h=h'$.
\end{proof}
\begin{cor}[Splitting principle]\label{cor:splitting_sheaf} Let $\varphi:G_1\to G_2$ be an isogeny, then 
\[
\varphi_\#:\cLog_{G_1}\to \varphi^!\cLog_{G_2}
\]
is an isomorphism. In particular, if 
$t:S\to G_1$ is in the kernel of $\varphi$, then
\[
t^*\cLog_{G_1}\isom\prod_{n\ge 0}\Sym^n\sH_{G_2}.
\]
\end{cor}
\begin{proof}
By Corollary \ref{thm:functoriality} the map $\gr\varphi_\#$ is 
an isomorphism as $\varphi_!:\sH_{G_1}\to \sH_{G_2}$ is already an isomorphism
(recall that we have $\bQ$-coefficients). From this one sees that 
$\varphi_\#:\Log{n}_{G_1}\to \varphi^!\Log{n}_{G_2}$ is an isomorphism. 
Applying $t^{!}$ gives, as $\varphi\circ t=e_2$, the isomorphism 
$t^!\cLog_{G_1}\isom t^!\varphi^!\cLog_{G_2}\isom(e_2)^!\cLog_{G_2}$.
By purity or more precisely Lemma \ref{lem:unipotent_uppershriek}
we get $t^*\cLog_{G_1}\isom(e_2)^{*}\cLog_{G_2}\isom \prod_{n\ge 0}\Sym^n\sH_{G_2}$.
\end{proof}
\subsection{Vanishing of cohomology}
The second property of the logarithm sheaf concerns the cohomology, which
is important for the proof of all other properties and
the definition of the polylogarithm.
\begin{theorem}[Vanishing of cohomology]\label{thm:vanishing}
One has 
\[
R^i\pi_!\cLog \isom \begin{cases}
\bQ(-d) & i=2d\\
0 & i\neq 2d.
\end{cases}
\]
Let $G$ be an extension of an abelian scheme of relative dimension $g$ by a torus or rank $r$. Then
$\sH$ is a locally constant $\bQ$-sheaf of dimension $h:=\dim_\bQ\sH=2g+r$, and one also has
\[
R^i\pi_*\cLog \isom \begin{cases}
{\bigwedge}^{h}\sH^{\vee} & i=h\\
0 & i\neq h
\end{cases}
\]
where $\sH^{\vee}=\underline{\Hom}_S(\sH,\bQ)$ is the dual of $\sH$.
\end{theorem}
The proof of this theorem will be given in Section \ref{proof:vanishing},
see Corollary \ref{cor:vanishing} and Corollary \ref{cor:vanishing_semiabelian}. 

The sheaf $\cLog$ can also be characterized by a universal property. Let $\sF$
be a unipotent sheaf of some finite length $n$ on $G$. Consider the homomorphism
\begin{equation}\label{eq:universal-property-map}
\pi_*\ul\Hom_G(\cLog,\sF)\to e^*\sF
\end{equation}
defined as the composition of 
\[
\pi_*\ul\Hom_G(\cLog,\sF)\to \pi_*e_*e^{*}\ul\Hom_G(\cLog,\sF)
\to \ul\Hom_S(e^*\cLog,e^*\sF)
\] 
with
\[
\ul\Hom_S(e^*\cLog,e^*\sF)\xrightarrow{(\mathbf{1})^{*}}
\ul\Hom_S(\bQ,e^*\sF)\isom e^*\sF.
\]
The same composition on the derived level defines a morphism
\begin{equation}\label{eq:universal-property-map-derived}
R\pi_*\ul {R\Hom}_G(\cLog,\sF)\to e^*\sF
\end{equation}

\begin{theorem}[Universal property]\label{thm:sheaf-univ-property}
Let $\sF$ be a unipotent sheaf, then the  map \eqref{eq:universal-property-map} induces an isomorphism
\[
\pi_*\ul\Hom(\cLog,\sF)\isom e^*\sF.
\]
Let $M$ be a unipotent object in the derived category of sheaves $\D(G)$.
Then the  morphism \eqref{eq:universal-property-map-derived} 
is an isomorphism
\[
R\pi_*\ul{R\Hom}(\cLog,M)\isom e^*M.
\]
\end{theorem}

As a consequence the functor $\sF\to\Gamma(S,e^*\sF)$ is pro-represented
by $\cLog$.

\begin{proof}
It suffices to treat the triangulated version. Indeed, if $M=\sF$ is
a sheaf, then $e^*\sF$ is concentrated in degree $0$, and hence
\[ R\pi_*\ul{R\Hom}(\cLog,\sF)=\pi_*\ul{\Hom}(\cLog,\sF).\]

We will show the theorem by induction on the length $n$ of
the unipotent object $M$. We start in the case $n=0$, $M=\pi^*N$. 
We claim that the natural map is  
an isomorphism
\[
R\pi_*R\ul{\Hom}_G(\cLog,\pi^*N)\isom N
\]
Writing $\pi^*N\isom \pi^!N(-d)[-2d]$ then one has by adjunction and because $R\pi_!\cLog\isom \bQ(-d)[-2d]$
\[
R\pi_*R\ul{\Hom}_G(\cLog,\pi^*N)\isom
R\underline{\Hom}_S(R\pi_!\cLog, N(-d)[-2d])\isom
R\underline{\Hom}_S(\bQ,  N)
\]
As $\underline{\Hom}_S(\bQ,  N)\isom N$ is the identity
functor, the claim follows. 

Now assume that the theorem is proven for unipotent objects of 
length $n-1$ and let $M$ be unipotent of length $n$. 
Then we have an exact triangle
\[
M'\to M\to M''
\]
with $M'$ and $M''$ unipotent of length less than $n$. 
We get a morphism of exact triangles
\[
\xymatrix{
R\pi_*R\ul{\Hom}_G(\cLog,M')\ar[r]\ar[d]^{\isom}&
R\pi_*R\ul{\Hom}_G(\cLog,M)
\ar[r]\ar[d]&R\pi_*R\ul{\Hom}_G(\cLog,M'')\ar[d] \\
e^*M'\ar[r]& e^*M\ar[r]& e^*(M'').
}
\]
By induction the outer vertical morphisms are isomorphisms, hence the same is
true in the middle.
\end{proof}

\section{Motivic Logarithm}\label{sec:motlog}
We work in the motivic setting described in Section \ref{sec:setting} and
the geometric situation described there. In particular,
let $S$ be noetherian and finite dimensional. Let $X\to S$ be separated and of finite type.
Recall that we work in the category 
$\DA(X)$ the triangulated category of \'etale motives
without transfers with {\em rational coefficients}, see Section \ref{sec:setting_motivic}.

\subsection{Motives of commutative group schemes}\label{sec:motive_cgs}
Let $G/S$ be a smooth commutative group scheme with connected fibres of relative dimension $d$. The group $G$ defines two natural \'etale sheaves of $\Q$-vector spaces
on the category of smooth $S$-schemes:
\begin{itemize}
\item
 on the one hand $T\mapsto \Q[G(T)]$; its image 
 in $\DA(S)$ is the motive $M_S(G)$.
\item on the other hand $T\mapsto G(T)\tensor\Q$. Following
 \cite[Definition 2.1, 2.3]{AHP} we write $\ul{G}_\Q$ for the \'etale sheaf and $M_1(G)$ for its image in $\DA(S)$. 
\end{itemize}
The summation map $\Q[G]\to\ul{G}_\Q$ induces a natural map
$M_S(G)\to M_1(G)$.

Let $\kd(G)$ be the Kimura dimension of $G$
(see \cite[Definition 1.3]{AHP}). It is at most $2d$.
 The main
result of \cite{AHP}  (see loc.cit. Theorem~3.3) is the existence of a decomposition
\begin{equation}\label{eq:motive-decomp} M_S(G)=\bigoplus_{i= 0}^{\kd(G)}M_n(G), \end{equation}
which is natural in $G$ and $S$.  Moreover,  we have
\[ M_n(G)=\Sym^nM_1(G)\]
and the isomorphism in \eqref{eq:motive-decomp} is
an isomorphism of Hopf objects. 
The motive $M_n(G)$ is
uniquely determined by naturality. 

%\annette{$a$-Operation nach hinten}\guido{ok}
%In particular, for all $a\in\Z$, the multiplication
%$[a]:G\to G$ operates as multiplication by $a^n$ on $M_n(G)$. In the language
%of Appendix \ref{appA}, the motive $M_S(G)$ has a finite decomposition
%into $[a]_*$-eigenspaces with eigenvalues 
%$a^0,a^1,\dots,a^{\kd(G)}$ for some (and hence all) $a\in\Z$.

By \cite[Section 5.2]{AHP} the image of $M_1(G)$ under the (covariant) $\ell$-adic realization is $\Hh_\ell[1]$ where 
$\Hh_\ell$ is the relative Tate-module of Definition \ref{defn:tate}. Its image under
the Betti-realization is the relative first homology $R^{-1}\pi_!\pi^!\Q[1]$.
This motivates the following definition:

\begin{defn}Let $G/S$ be a smooth commutative group scheme with connected fibres. Let
$\Hh:=\Hh_{G/S}\in\DA(S)$ be defined as $M_1(G)[-1]$.
\end{defn}

\subsection{Kummer motives}

\begin{defn}Let $G/S$ be a smooth commutative group scheme with connected fibres. 
Let $s:S\to G$ be a section. 
The {\em Kummer motive} $\Kh(s)$ given by $s$ is
the image of the complex of \'etale sheaves
\[[\Q_S\xrightarrow{s} \ul{G}_\Q]\]
(with $\Q_S$ in degree $0$) in the category $\DA(S)$. 
The {\em Kummer extension} of $s$ is the natural triangle
\[ \Hh_{G/S}\to \Kh(s)\to \Q_S\xrightarrow{s}\Hh_{G/S}[1]\ .\]
\end{defn}

This defines a natural group homomorphism (the {\em motivic Kummer map})
\[ G(S)\to \Hom_{\DA(S)}(\Q_S,\Hh_{G/S}[1])\ .\]
It maps the unit section to the trivial extension. More precisely,
$\Kh(e)$ is the image of the complex of \'etale sheaves
$[\Q_S\xrightarrow{0}\ul{G}_\Q]$, hence the natural inclusion
$[\Q_S\to 0]\to [\Q_S\xrightarrow{0}\ul{G}_\Q]$ induces 
a distinguished splitting
\[ \Kh(e)=\Q_S\oplus \Hh_{G/S}[-1]\]
\begin{rem}
It may seem strange at first glance that the motivic extension $\Log{1}$ has a distinguished splitting, whereas the $\Log{1}$ sheaf has not. In fact, there is a unique splitting of the sheaf theoretic version of $\Log{1}$, which is compatible with all isogenies (see \cite[Section 1.5.]{BKL} for an elaboration).
This splitting coincides with the motivic splitting under the realizations. 
\end{rem}
\begin{lemma}\label{lem:describe_kummer}The Kummer extension is given by the projection
\[ M_S(S)\xrightarrow{M_S(s)} M_S(G)=\bigoplus_i M_i(G)\to M_1(G)=\Hh_{G/S}[1]\]
under the decomposition of \cite{AHP}.
\end{lemma}
\begin{proof}
By construction in loc.cit. the map $M_S(G)\to M_1(G)$ is induced
from the morphism of \'etale sheaves $\Q[G]\to\ul{G}_\Q$. Also
by construction $s:M_S(S)\to M_S(G)$ is induced from
$s:\Q_S=\Q[S]\to \Q[G]$. Hence the composition is induced from
$\Q_S\to\ul{G}_\Q$.
\end{proof}

\begin{rem}\label{lem:compare_kummer}
Let $\ell$ be a prime invertible on $S$.
Then the realization of the Kummer extension is the $\ell$-adic Kummer extension
\[ 0\to \Hh\to\Kh(s)_l\to\Q_l\to 0\]
in $\Ext^1_S(\Q_l,\Hh)$. We do not go into details because we will not need this fact.
\end{rem}
%\begin{proof}\TODO (weglassen)universal case suffices, i.e., $\Log{1}$. Splitting ok.
%Leray spectral sequence: see proof of 3.5.2 (4)
%\end{proof}

\subsection{Logarithm sheaves}
Let $G/S$ be smooth commutative group scheme with connected fibres.

\begin{defn}\label{defn:logn_motivic} Consider $G\times_SG\to G$ via the first projection.
Let $\Delta:G\to G\times G$ be the diagonal. We put
\[\Log{1}=\Kh(\Delta)\in\DA(G)\]
together with the splitting $\one{1}:\bQ\to e^*\Log{1}$ given
by $e^*\Kh(\Delta)=\Kh(e)$ as before.

We define
\[ \Log{n}=\Sym^n\Log{1}\]
and denote by $\one{n}$ the induced splitting $\Sym^n(\one{1}):\Q_S\to \Log{n}$.
\end{defn}

We first establish the basic properties analogous to the sheaf theoretic case.

\begin{lemma}\label{lem:log_in_e}
The section $\one{n}$ induces isomorphisms
$e^*\Log{n}\to \bigoplus_{i=0}^n\Sym^i\Hh$ and
$e^!\Log{n}\to\bigoplus_{i=0}^n\Sym^i\Hh(-d)[-2d]$.
\end{lemma}
\begin{proof}The case $n=1$ was discussed above. Passing to 
symmetric powers, we get 
\[\Sym^n\Log{1}\isom \bigoplus_{i=0}^n\Sym^i\Q_S\tensor\Sym^{n-i}\Hh\]
as claimed. The statement on $e^!\Log{n}$ follows by Lemma \ref{lem:unipotent_uppershriek}.
\end{proof}

\begin{prop}\label{log-triangle}For $n\geq 1$ there is a  system of exact triangles in $\DA(G)$:
\[ \Sym^n\pi^*\Hh_{G/S}\to\Log{n}\to \Log{n-1}\ .\]
\end{prop}
\begin{proof} Consider first the case $n=1$. By definition, we have a
distinguished triangle
\[ \Hh_{G\times G/G}\to \Log{1}\to \Q_S\ .\]
By compatibility of $M_1(G)$ with pull-back (see  \cite[Proposition~2.7]{AHP}) we have
\[ \pi^*M_1(G/S)=M_1(G\times G/G)\ .\]
This finishes the proof in this case. We abbreviate $\Hh$ for both
$\Hh_{G/S}$ and $\pi^*\Hh_{G\times G/G}$.

Recall that $\Log{n}$ is the image of a complex $\ul{\cLog}^{(n)}$ of \'etale sheaves on $G$.
The complex $\ul{\cLog}^{(1)}$ has a filtration
\[ 0\to \pi^*\Hh\to \ul{\cLog}^{1}\to \Q_G\to 0\]
in the abelian category of complexes of \'etale sheaves.
Hence the symmetric powers also have a natural filtration (for full details see \cite{AEH} Appendix C). Its associated gradeds are 
\[ \Sym^i(\Hh)\tensor\Sym^j\Q_G=\Sym^i\Hh\ .\]
In the same way as in the $\ell$-adic case, see the discussion before Definition \ref{defn:logn}, we get short exact sequences
of complexes of sheaves
\[ 0\to\Sym^n\Hh\to \ul{\cLog}^{(n)}\to\ul{\cLog}^{(n-1)}\to 0\ .\]
We view them as triangles in $\DA(G)$.
\end{proof}

\subsection{Functoriality}

\begin{theorem}\label{log-funct}Let $\varphi:G_1\to G_2$ be morphism of smooth group schemes with connected fibres over $S$. Let $c=d_1-d_2$ be the relative fibre dimension.
Then there is  a natural map
\[ \varphi_\#:\Log{n}_{G_1}\to \varphi^*\Log{n}_{G_2}=\varphi^!\Log{n}_{G_2}(-c)[-2c].\]
\end{theorem}
\begin{proof}
We construct the map to $\varphi^*\Log{n}_{G_2}$. 
By Lemma \ref{lem:unipotent_uppershriek} one has $\varphi^*\Log{n}_{G_2}=\varphi^!\Log{n}_{G_2}(-c)[-2c]$. As $\varphi^*$ commutes
with tensor product, it
suffices to treat the case $n=1$.
We have  the commutative diagram
\[\begin{CD}
 G_1@>\Delta>>G_1\times G_1\\
 @V\varphi VV@VV\varphi V\\
 G_2@>\Delta>>G_2\times G_2
\end{CD}\]
i.e., $\Delta_{G_1}\in G_1\times G_1(G_1)$ is mapped to $\Delta_{G_2}\in G_2\times G_2$.
This implies that the diagram of sheaves on $G$ commutes
\[ \begin{CD} \Q_{G_1}@>>> \ul{G_1\times G_1}_\Q=\pi_{G_1}^*\ul{G}_{1\Q}\\
@V\id VV @VV\pi^*_{G_1} \varphi V \\
    \varphi^*\Q_{G_2}@>>> \varphi^*\ul{G_2\times G_2}_\Q=\pi_{G_2}^*\ul{G_2}_\Q
\end{CD}\]
We take the image of this diagram in $\DA(G_1)$. The statement follows
because $\varphi^*M_1(G_2)=M_1(\varphi^*G_2)=M_1(G_1)$ by \cite[Proposition 2.7]{AHP}.
\end{proof}

\begin{cor}[Splitting principle] \label{lem:log_in_torsion}
Let $\varphi:G_1\to G_2$ be an isogeny, then 
\[
\varphi_\#:\Log{n}_{G_1}\to \varphi^!\Log{n}_{G_2}
\]
is an isomorphism. In particular, if 
$t:S\to G_1$ is in the kernel of $\varphi$, then
\[
t^*\cLog_{G_1}\isom\prod_{n\ge 0}\Sym^n\sH_{G_2}.
\]
\end{cor}
\begin{proof}
As $\varphi^*$ is compatible with tensor product and exact triangles, it
suffices to show $\varphi^*\Hh_{G_2}=\Hh_{G_1}$ or equivalently $\Hh_{G_2}=\Hh_{G_1}$ as motives on $S$. This holds by construction because $\ul{G}_{2\Q}=\ul{G}_{1\Q}$.
The rest of the argument is the same as in the sheaf theoretic case,
see Corollary~\ref{cor:splitting_sheaf}.
%Consider $[N]:G\to G$. This is an isogeny. By Lemma \ref{log-funct}
%we have $[N]^*\Log{n}=\Log{n}$. Hence by Lemma \ref{lem:log_in_e}
%\[ e_N^*\Log{n}=e_N^*[N]^*\Log{n}=\pi_N^*e^*\Log{n}=\pi_N^*\left(\bigoplus_{i=0}^n\Sym^i\Hh\right)=\bigoplus_{i=0}^n\Sym^i\pi_N^*\Hh.\]
%If $N$ is invertible on $S$, then $G[N]\to S$ is etale and second statement follows by Lemma \ref{lem:unipotent_uppershriek}.
\end{proof}

\subsection{Vanishing of cohomology}
The second property of the logarithm sheaf concerns the vanishing of the cohomology, which
is important for the proof of all other properties and
the definition of the polylogarithm.
\begin{theorem}[Vanishing of cohomology]\label{thm:vanishing-motivic}
Assume that $S$ is a scheme of characteristic $0$ or that $G/S$ is affine.
One has 
\[
R\pi_!\cLog \isom \bQ(-d)[-2d] \
\]
\end{theorem}
The proof of this theorem will be given in Section \ref{sec:motivic_proof}. 

As in the sheaf theoretic case, this implies a universal property of the motivic logarithm. Let $M$
be a unipotent sheaf of length $n$ on $G$. In the same way as in the case of sheaves (see equation \eqref{eq:universal-property-map}) one has a map 
\begin{equation}\label{eq:motivic-univ-property-map}
R\pi_*R\ul{\Hom}_G(\cLog,M)\to e^*M.
\end{equation}
\begin{theorem}[Universal property]\label{thm:univ-property}
Let $S$ be a scheme of characteristic $0$ or assume that $G/S$ is affine.
Let $M$ be a unipotent motive on $G$, then
the map
\eqref{eq:motivic-univ-property-map} induces an isomorphism
\[
R\pi_*R\ul\Hom(\cLog,M)\isom e^*M.
\]
\end{theorem}
\begin{proof}The argument is the same as in the sheaf theoretic case, with
Theorem \ref{thm:vanishing-motivic} replacing Theorem \ref{thm:vanishing}.
\end{proof}

\subsection{Realizations}

\begin{prop}\label{prop:real}
\begin{enumerate}
\item
Assume the prime $\ell$ is invertible on  $S$ and $S$ of finite over a regular scheme of dimension $0$ or $1$. Then
the $\ell$-adic realization $R_\ell$ maps the motivic $\Log{n}_G$ to the $\ell$-adic $\Log{n}_G$ as defined in Section \ref{sec:log}
.
\item Assume $S$ is of finite type over $\C$. Then the Betti realization $R_B$ maps the motivic $\Log{n}_G$ to the constructible $\Log{n}_G$ in Section \ref{sec:log}
\end{enumerate}
\end{prop}
\begin{proof}
The argument is the same in both cases. By construction it suffices to consider the case $n=1$. 
We use the description of the Kummer extension for $\Delta$ given
in Lemma \ref{lem:describe_kummer}. After applying the realization
functor (which commutes with all 6 functors), we obtain
the same class as constructed in Equation (\ref{eq:comp}). By Lemma~\ref{lem:compute_once} this is $\Log{1}$ in the realization.
\end{proof}

\begin{rem}The same argument will also apply in the Hodge theoretic setting once we have a realization functor compatible with the 6 functor formalism.
See the discussion in Section \ref{sec:setting_motivic} on the state of the art.
\end{rem}

\section{The polylogarithm sheaf/motive}\label{sec:polylog}
Unless stated otherwise, we work in the sheaf theoretic and in the motivic setting in parallel.
The pro-sheaf $\cLog=(\Log{n})_{n\geq 0}$ is the one of Definition
\ref{defn:logn} and Definition \ref{defn:logn_motivic}, respectively.

\subsection{Residue sequences}
As before let $\iota_D:D\to G$ be a closed subscheme which is \'etale over $S$ and contained in some scheme of torsion points $G[N]$. Of particular interest is the case $D=e(S)$.
Recall the localization triangle attached to $j_D:U_D\to X\leftarrow D:\iota_D$.
For any $\sF$ it defines a connecting morphism
\[  R\pi_!Rj_{D*}j_D^*\sF[-1]\to R\pi_!\iota_{D!}\iota_D^!\sF=\pi_{D!}\iota_D^!\sF\ .\]
We apply this to $\sF=\Log{n}(d)[2d]$. This is unipotent, so by Lemma
\ref{lem:unipotent_uppershriek}, we may replace $\iota_D^!$ by $\iota_D^*$.
Moreover, recall the sheaf theoretic and motivic splitting principles \ref{cor:splitting_sheaf} and Lemma \ref{lem:log_in_torsion}, respectively. Together
we have a canonical identification
\[ \pi_{D!}\iota_D^!\Log{n}(d)[2d]\isom \bigoplus_{i=0}^n\pi_{D!}\Sym^i\pi_D^*\Hh.\]

\begin{defn}\label{defn:res_motivic}
The composition of the  above morphisms
\[ R\pi_!Rj_{D*}j_D^*\Log{n}(d)[2d-1]\to \pi_{D!}\iota_D^!\Log{n}(d)[2d]=\bigoplus_{i=0}^n\pi_{D!}\Sym^i\pi_D^*\Hh\]
is called {\em residue map at $D$}.
\end{defn}

The residue triangle also induces a connecting homomorphism, also called residue map,
\[  \Ext_S^{2d-1}(\sF,R\pi_!Rj_{D*}j_D^*\Log{n}(d))
\to \Hom_S(\sF,\bigoplus_{i=0}^n\pi_{D!}\Sym^i\pi_D^*\Hh).\]

\begin{lemma} [Functoriality]\label{lem:res_functorial}The residue map is functorial. More precisely, let
$\phi:G_1\to G_2$ be a morphism of smooth group schemes with connected fibres
over $S$. Let $D_1\subset G_1$ and $D_2\subset G_2$ be closed subschemes
\'etale over $S$ such that $\phi(D_1)\subset D_2$. Then the morphism 
\[ \phi_\#:\Log{n}_{G_1}(d_1)[2d_1]\to\varphi^{!}\Log{n}_{G_2}(d_2)[2d_2]\]
of Theorem~\ref{thm:functoriality} and Lemma \ref{log-funct}, respectively, induces a morphism of exact triangles
\[\begin{CD}
 R\phi_!\Log{n}_{G_1}(d_1)[2d_1-1]@>>>R\phi_!Rj_{D_1*}j_{D_1}^{*}\Log{n}_{G_1}(d_1)[2d_1-1]@>>>  R\phi_!\iota_{D_1!}\bigoplus_{i=0}^n\Sym^i\pi_{D_1}^*\Hh_{G_1}\\
@VVV@V\varphi_\#VV@VVV\\
 \Log{n}_{G_2}(d_2)[2d_2-1]@>>>Rj_{D_2*}j_{D_2}^{*}\Log{n}_{G_2}(d_2)[2d_2-1]@>>>  \iota_{D_2!}\bigoplus_{i=0}^n\Sym^i\pi_{D_2}^*\Hh_{G_2}
\end{CD}\]
\end{lemma}
\begin{proof}Let $c$ be the relative dimension of $G_1$ over $G_2$ and denote by $U_{D_i}$ the complement of $D_i$ and
by $U_{\varphi^{-1}D_2}\subset U_{D_1}$ the complement of 
$\varphi^{-1}D_2$.
We apply $j_{D_1*}j_{D_1}^*$ to $\phi_\#$ and restrict
to $U_{\varphi^{-1}D_2}$and obtain
\[ j_{D_1*}j_{D_1}^*\Log{n}\to j_{D_1*}j_{D_1}^*\phi^!\Log{n}(c)[2c]\to j_{\phi^{-1}D_2*}j_{\phi^{-1}D_2}^*\phi^!\Log{n}(c)[2c].\]
We have a cartesian square
\[\begin{CD}
U_{\varphi^{-1}D_2}@>j_{\phi^{-1}D_2}>> G_1\\
@V\phi VV@VV\phi V\\
U_{D_2}@>j_{D_2}>> G_2
\end{CD}\] 
which implies $j_{D_2}^*\phi^!=\phi^!j_{\varphi^{-1}D_2}^*$. Together with the 
base change $Rj_{\phi^{-1}D_2*}\phi^!=\phi^!Rj_{D_2*}$  this gives a map
\[ Rj_{D_1*}j_{D_1}^*\Log{n}\to \phi^!Rj_{D_2*}j_{D_2}^{*}\Log{n}(c)[2c]\]
or equivalently
\[  R\phi_!Rj_{D_1*}j_{D_1}^*\Log{n}(d_1)[2d_1-1]\to Rj_{D_2*}j_{D_2}^*\Log{n}(d_2)[2d_2-1]\]
The analogous argument for $\iota_{D_1!}\iota_{D_1}^!$ gives
\[ R\phi_!\iota_{D_1!}\iota_{D_1}^!\Log{n}(d_1)[2d_1]\to \iota_{D_2!}j\iota_{D_1}^!\Log{n}(d_2)[2d_2].\]
This defines a morphism of exact triangles. We now apply
the identification via the splitting principle on $D_1$ and $D_2$.
\end{proof}

\subsection{The main result}
We formulate all results on polylog in two big statements. We keep the notation and the setting of Section \ref{sec:setting}. 
\begin{theorem}[Polylog with respect to the unit section]\label{polylog_all}
Let $S$ be a base scheme satisfying the assumptions
of the respective setting, see Section \ref{sec:setting}. Let
$G/S$ be a smooth commutative $S$-group scheme with connected fibres of dimension $d$.
\begin{enumerate}
\item There is a unique system of classes
\[ \pol{n}\in\Ext^{2d-1}_S(\Hh_G, R\pi_!Rj_*j^*\Log{n}_G(d))\]
such that
\begin{enumerate}
\item their residue in $e^!\Log{n}_G(d)[2d]\isom \bigoplus_{i=0}^n\Sym^i\Hh_G$ is
the natural inclusion of $\Hh_{G}$;
\item they are compatible under the transition maps $\Log{n+1}_G\to\Log{n}_G$;
\item they are functorial with respect to homomorphisms of
groups schemes $\varphi:G_1\to G_2$, i.e., the diagrams
\[\begin{CD}
     \Hh_{G_1}@>\pol{n}_{G_1}>> R\pi_{1!}Rj_{1*}j_1^*\Log{n}_{G_1}(d_1)[2d_1-1]\\
     @V\varphi_!  VV @VV\varphi_\# \text{\ (\ref{lem:res_functorial})} V\\
 \Hh_{G_2}@>\pol{n}_{G_2}>> R\pi_{2!}Rj_{2*}j_2^*\Log{n}_{G_2}(d_2)[2d_2-1]\\
\end{CD}\]
commute.
\end{enumerate}
\item The classes $\pol{n}$ are contravariantly functorial
 under morphisms $S'\to S$.
\item If $\ell$ is invertible on $S$ which is of finite type over a regular scheme of dimension $0$ or $1$, then the motivic class is mapped to
the $\ell$-adic class by the $\ell$-adic realization functor $R_\ell$.
\item If $S$ is of finite type over $\C$, then the motivic class is mapped to
the analytic class by the Betti-realization functor $R_B$.
\end{enumerate}
\end{theorem}

Let $D\subset G$ be a closed subscheme which is \'etale over $S$ and contained in $G[N]$ for some
$N$.
\begin{defn}
Let
\[ 
\bQ[D]^0:=\ker\left(H^0(S,\pi_{D!}\bQ)\to H^0(S,\bQ)\right),
\]
where $\pi_{D!}\bQ\to \bQ$ is the trace map.
\end{defn}
This should be thought of as $\bQ$-valued functions $f$ on $D$ with
$\sum_{d\in D}f(d)=0$, which is literally true in the case where $D$ is
a disjoint set of sections. 

Note that by the isomorphism $\pi_{D!}\iota_D^!\Log{n}_G(d)[2d]\isom \pi_{D!}\bigoplus_{i=0}^n\Sym^i\Hh_G$ induced by the splitting principle, one has an inclusion 
\[
\bQ[D]^{0}\subset \ker\left( H^{0}(S, \pi_{D!}\iota_D^!\Log{n}_G)\to H^{0}(S,\bQ)\right).
\]
Let $\varphi:G_1\to G_2$ is a homomorphism of smooth group schemes with connected fibres, $D_1\subset G_1$ and $D_2\subset G_2$ as above such that $\phi(D_1)\subset D_2$. Then the trace map also
induces
\[ \phi_!:\bQ[D_1]^{0}\to \bQ[D_2]^{0}.\]

\begin{theorem}[Polylog with respect to a subscheme]\label{thm:polylog_divisor}
Let $S$ be a base scheme satisfying the assumptions
of the respective setting, see Section \ref{sec:setting}. Let
$G/S$ be a smooth $S$-group scheme with connected fibres of dimension $d$.
Let $D\subset G$ be a closed subscheme which is \'etale over $S$ and contained in $G[N]$ for some
$N$ and \'etale.
Let $\alpha\in \bQ[D]^0$.
\begin{enumerate}
\item There is a unique system of classes
\[\pol{n}_\alpha\in\Ext^{2d-1}(\bQ,R\pi_!j_{D*}j_D^*\Log{n}(d))\]
such that
\begin{enumerate}
\item their residue in $\ker\left( H^{0}(S, \pi_{D!}\iota_D^!\Log{n}_G)\to H^{0}(S,\bQ)\right)$
is given by $\alpha$;
\item they are compatible under the transition maps 
$\Log{n+1}_G\to\Log{n}_G$;
\item they are functorial with respect to  homomorphism of group schemes
$\varphi:G_1\to G_2$ mapping $D_1\subset G_1$ into $D_2$, i.e., the class $\pol{n}_\alpha$ is mapped to
$\pol{n}_{\phi_!\alpha}$ under the map 
\[ \phi_\#:\Ext^{2d_1-1}_S(\bQ,R\pi_{1!}j_{D*}j_D^*\Log{n}_{G_1}(d_1))\to \Ext_S^{2d_2-1}(\bQ,R\pi_{2!}j_{\phi{D}*}j_{\phi(D)}^*\Log{n}_{G_2}(d_2))\]
induced from Lemma \ref{lem:res_functorial}.
\end{enumerate}
\item The classes $\pol{n}_\alpha$ are contravariantly functorial
 under morphisms $S'\to S$.
\item If $\ell$ is invertible on $S$ which is of finite type over a regular scheme of dimension $0$ or $1$, then the motivic class is mapped to
the $\ell$-adic class by the $\ell$-adic realization functor $R_\ell$.
\item If $S$ is of finite typer over $\C$, then the motivic class is mapped to
the analytic class by the Betti-realization functor $R_B$.
\end{enumerate}
\end{theorem}
\begin{rem}The proof of the theorems are nearly identical and will be given together. 
We are going to give two different arguments: 
\begin{itemize}
\item The first proof uses the cohomological vanishing of Theorem \ref{thm:vanishing}. It has the advantage
of being quick and direct.
The argument is valid in the sheaf theoretic setting and relies on the fact that the polylogarithm classes
for $G$ are uniquely determined by their residues and compatibility
with respect to $n$. It also applies in the motivic setting under the more
restrictive assumptions of Theorem \ref{thm:vanishing-motivic}. \item
The second proof is valid in any setting and relies on the fact that the polylogarithm classes for $G$ are uniquely determined by their residues and uses the functoriality with respect to
multiplication $[a]:G\to G$ for a single $a\in\Z, a\neq 0, \pm 1$ (satisfying 
$[a]^*D\subset D$ in the case of polylog with respect to a divisor).
Indeed, they are going to be characterized as the unique 
preimages of their residues on which $[a]$ operates by multiplication
by 
$a^1$ and $a^0$, respectively.
\end{itemize}
\end{rem}

\begin{rem}The argument for compatibility with realizations will
also apply in Hodge theoretic setting once a Hodge realization
functor compatible with the six functor formalism is constructed.
This is not yet the case, see the discussion at the end of Section \ref{setting_realizations} for the state of the art.
\end{rem}

\begin{rem}In the simplest case $G=\Gm$, the above class is not the same
as the one in the literature, but rather maps to it.
See Section \ref{sec:with_*}  for the precise relation.
\end{rem}

\subsection{First proof}
We work in the sheaf theoretic setting. The same arguments also
apply in the motivic setting if the characteristic is $0$ or if $G/S$
is affine.

Recall that by Theorem \ref{thm:vanishing} and Theorem \ref{thm:vanishing-motivic}, respectively, we have
\[ R\pi_!\cLog(d)[2d]=\bQ.\]

%\begin{prop}\label{prop:ext-comp} Let $\sF$ be a sheaf on $S$. Then
%the localization sequence with respect to $D$ induces 
%an exact sequence
%\[
%0\to \Ext^{2d-1}_{S}(\sF,R\pi_!Rj_{D*}j_D^{*}\cLog(d))\to
%\Hom_S(\sF, \pi_{D!}\iota_D^*\cLog)\to \Hom_S(\sF,\bQ)
%\]
%\end{prop}
%\begin{proof} Consider the localization triangle
%\[
%\iota_{D!}\iota_D^{!}\cLog(d)\to \cLog(d)\to Rj_{D*}j_D^{*}\cLog(d)
%\]
%and use relative purity (Lemma \ref{lem:unipotent_uppershriek}) $\iota^{!}_D\cLog(d)\isom \iota^{*}_D\cLog[-2d]$. Then apply $R\pi_!$ and use the
%computation $R\pi_!\cLog(d)\isom \bQ[-2d]$ of Theorem \ref{thm:vanishing}
%(or Theorem \ref{thm:vanishing_motivic} in the motivic case), in order to get
%\[ R\pi_{D!}\iota_D^{*}\cLog[-2d]\to \bQ[-2d]\to Rj_{D*}j_D^{*}\cLog.\]
%Applying $\Hom_S(\sF,-[2d-1])$ one gets the desired exact sequence
%\[
%0\to \Ext^{2d-1}_{S}(\sF,R\pi_!Rj_{D*}j_D^{*}\cLog(d))\to
%\Hom_S(\sF, \pi_{D!}\iota_D^*\cLog)\to \Hom_S(\sF,\bQ).
%\]
%\end{proof}
\begin{prop}\label{prop:ext-comp} 
We work either in the sheaf theoretic setting or
the motivic setting with $S$ of characteristic $0$ or $G/S$ affine.
Let $\sF=\sH$ or $\sF=\bQ$.  
There is an exact sequence
\[ 0\to \Ext^{2d-1}_{S}(\sF,R\pi_!Rj_{D*}j_D^{*}\cLog(d))\xrightarrow{\res}
\Hom_S(\sF, \pi_{D!}\iota_D^*\cLog)\to \Hom_S(\sF,\bQ).
\]
where the last map is the composition of  the augmentation 
$\pi_{D!}\iota_D^*\cLog\to\pi_{D!}\iota_D^* \bQ$ and the 
the trace map $\pi_{D!}\iota_D^*\bQ\to \bQ$.
\end{prop}
\begin{proof}
We apply $R\pi_!$ and $\Hom_S(\sF,-)$ to the localization triangle
and using the computation of $R\pi_!\cLog(d)[2d]$. 

It remains to show that $\Hom_S(\sF, \bQ)$ vanishes for
$\sF=\sH$ and $\sF=\bQ$. This is clear in the sheaf theoretic setting because
negative Ext-groups vanish.

We now turn to the motivic setting. If $\sF=\bQ$,
the vanishing of $\Hom_S(\bQ,\bQ[-1])$ is \cite[Proposition~11.1]{ayoubet}. 
If $\sF=\sH$, then
\begin{multline*} 
\Hom_S(\sH[1],\bQ)=\Hom_S(M_1(G),\bQ)\subset 
\Hom_S(M_S(G),\bQ)=\Hom_G(\bQ,\bQ)=\bQ\end{multline*}
again by  \cite[Proposition~11.1]{ayoubet}. The morphism
$\Hom_S(\bQ,\bQ)\to\Hom_G(\bQ,\bQ)$ is an isomorphism, hence the direct summand
$\Hom_S(M_1(G),\bQ)$ vanishes.
\end{proof}

\begin{proof}[Proof of Theorem \ref{polylog_all} and Theorem \ref{thm:polylog_divisor}.]
We first apply Proposition \ref{prop:ext-comp} with $\sF=\sH$ and $D=e(S)$.
We obtain the exact sequence
\[ 0\to \Ext^{2d-1}_{S}(\sH,R\pi_!Rj_{*}j^{*}\cLog(d))\to
\Hom_S(\sH, \prod_{i=0}^\infty\Sym^i\sH)\to \Hom_S(\sH,\bQ).\]
with the last map induced from the natural projection.
We define 
\[\cpol\in \Ext^{2d-1}_S(\sH,R\pi_!Rj_{*}j^{*}\cLog(d))\]
as the preimage of the natural inclusion of the $\sH$ into
$\prod_{i=0}^\infty\Sym^i\sH$. This means we have defined a
system of elements
\[ \pol{n}\in \Ext^{2d-1}_S(\sH,R\pi_!Rj_{*}j^{*}\Log{n}(d))\]
compatible under transition maps. It is uniquely determined
by these properties.

We now turn to functoriality under $\phi:G_1\to G_2$. By functoriality,
$\cpol_{G_1}$ and $\cpol_{G_2}$ both define elements in
$\Ext_S^{2d_2-1}(\sH_{G_1},R\pi_{2!}Rj_{2*}j_2^*\cLog_{G_2}(d_2))$
with the same residue in 
$\Hom_S(\sH_{G_1},\prod_{i=0}^\infty\Sym^i\sH_{G_2})$. By Proposition
\ref{prop:ext-comp} this implies that they agree.

The  behaviour under realizations
follows from these properties for $\cLog$ (see Proposition \ref{prop:real}) and uniqueness.

In the case of $\cpol_\alpha$, we obtain the sequence
\[ 0\to \Ext^{2d-1}_{S}(\bQ,R\pi_!Rj_{D*}j_D^{*}\cLog(d))\to
\Hom_S(\bQ, \prod_{i=0}^\infty\pi_{D!}\Sym^i\pi_D^*\sH)\to \Hom_S(\bQ,\bQ).\]
By assumption $\alpha$ is in the kernel of the last map. We define
$\cpol_\alpha$ as its preimage. 
All other argument are the same as in the case of pol with respect to
the unit section.
\end{proof}

\subsection{Second proof}
We work in the sheaf theoretic and in the motivic setting in parallel.
The argument relies on analysing the eigenspace decomposition
under the operation of multiplication by $a\in \Z$ on $G$.
Let $[a]:G\to G$ be the morphism on $G$.

Recall
that an $[a]$-linear operation on an object $X\in\D(G)$ is
the datum of a morphism $X\to [a]^!X$ or equivalently $f_a:[a]_!X\to X$.
By naturality it induces a map $\pi_!X=\pi_![a]_!X\xrightarrow{\pi_!f_a}\pi_!X$.

Such an $[a]$-linear operation on $\Log{n}$ was defined in Theorem \ref{thm:functoriality} and Theorem~\ref{log-funct}, respectively.

Recall also from Appendix \ref{appA} the notion of a finite decomposition
into generalized $[a]$-eigenspaces in a $\Q$-linear triangulated category.

\begin{prop}\label{eigenweights}
Let $a\in\Z$. 
\begin{enumerate}
\item
Then $R\pi_!\bQ$ has a finite decomposition into
$a$-eigenspaces
\[ R\pi_!\bQ=\bigoplus_{i=0}^{\kd(G)}\Sym^i\sH(-d)[i-2d]\]
with $a$ operating on $\Sym^{i}\sH$ by multiplication by $a^i$.
\item  Let $n\geq 0$. Under the operation of $[a]$ on the associated graded of $\Log{n}$, the object $R\pi_!\pi^*\Sym^n\Hh$ on $S$ has a finite decomposition into 
$[a]$-eigenspaces with eigenvalues $a^{n},\dots,a^{n+\kd(G)}$.
\item The object $R\pi_!\Log{n}$ on $S$ has a finite decomposition into generalized
$[a]$-eigenspaces with eigenvalues $a^0,\dots, a^{n+\kd(G)}$.
\item For $n\geq 1$ the map $R\pi_!\Log{n}\to R\pi_!\Log{n-1}$ induces an
isomorphism on $a^0$-eigenspaces. In particular, this eigenspace is isomorphic
to $\Q_S(-d)[-2d]$. 
\item For $n\geq 1$, the $a^1$-eigenspace of $R\pi_!\Log{n}$ vanishes.
\end{enumerate}
The decompositions are independent of the choice of $a$.
\end{prop}
\begin{proof}
We have the
formula
\[
R\pi_!\bQ=R\pi_!\pi^!\bQ(-d)[-2d]=\bigoplus_{i=0}^{\kd(G)}\Sym^i\Hh(-d)[i-2d]
\]
hence it suffices to
show that $[a]$ operates as multiplication by $a$ on $\Hh$.
The motivic case is established
in \cite[Theorem~3.3.]{AHP} (it follows directly from the description of $M_1(G)$
as the motive induced by $\ul{G}_\Q$).
The sheaf theoretic case is classical. It also follows immediately
from the motivic case and compatibility under realizations. This finishes
the proof of the first claim.

By Theorem \ref{thm:functoriality} and Theorem \ref{log-funct}, the
operation of $[a]$ on $\pi^*\Sym^n\Hh$ under the functoriality of
$\Log{n}$ is given by $\Sym^n[a]_!=a^n$. 
By the projection formula
\[ R\pi_!\pi^*\Sym^n\Hh=(R\pi_!\bQ)\tensor\Sym^n\Hh.\]
Hence the second statement follows from the first.

%$\Sym^n\Hh[n]$ on $G$ is a direct summand of 
%$\Hh^{\tensor n}[n]=M_1(G/S)^{\tensor n}$. It is
%a direct summand of the $n$-th J\"unneth component
%$M_n(G^{n}/S)$. By the projection formula
%\begin{multline*} \pi_!(\pi^*M_n(G^n)\tensor\Q)=M_n(G^n)\tensor\pi_!\Q\\
%=M_n(G^n/S)\tensor M_S(G)(-d)[-2d]=\bigoplus_{i=0}^{\kd(G)}M_n(G^n/S)\tensor M_i(G)(-d)[-2d].
%\end{multline*}
%We now need to understand the operation of $[a]$. 
For the third assertion, consider the exact triangle
\[ R\pi_!\Sym^{n}\Hh\to R\pi_!\Log{n}\to R\pi_!\Log{n-1}\ .\]
By induction and Proposition \ref{triangulated-eigenspaces} , we get a decomposition for $R\pi_!\Log{n}$ with eigenvalues as stated. Passing to
the $a^0$-eigenspace preserves exact triangles by the same Proposition \ref{triangulated-eigenspaces}.
There is no contribution from $R\pi_!\Sym^n\Hh$ for $n\geq 1$. In the case
$n=0$, the contribution is the component $i=0$ 
in assertion (1).
%\[R\pi_!\pi^*\bQ=R\pi_!\pi^!\bQ(-d)[-2d]=\bigoplus_{j=0}^{\kd{G}}\Sym^j\Hh(-d)[j-2d]\ .\]

We now consider the generalized eigenspace for the eigenvalue $a^1$. There is
no contribution from $R\pi_!\Sym^n\Hh$ for $n\geq 2$. Hence it suffices
to show the vanishing for $n=1$. We pass to the $a^1$-eigenspace in
the triangle for $n=1$ and have
\[ \Hh\tensor \bQ(-d)[-2d]\to\  ?\ \to \Hh(G)(-d)[1-2d].\]
It remains to show that the connecting morphism 
%\[ M_1(G)(-d)[-2d]\to \Hh\tensor M_0(G)(-d)[-2d][1]=M_1(G)\tensor M_0(G)(-d)[-2d]
%\]
is the identity. In the sheaf theoretic case, this is true by definition
of $\Log{1}$, see Definition \ref{defn:log}. In the motivic case, this
was checked during the proof of Proposition \ref{prop:real} on compatibility of the motivic logarithm with realizations.

Let $a\neq b$ be integers. Note that $[a]$ and $[b]$ commute. By Lemma
\ref{simultan}, the object $\Log{n}$ has a simultaneous decomposition
into generalized eigenspaces with respect to both. We show inductively that
the generalized eigenspaces for $a^i$ and $b^i$ agree
from the same statement for $\Sym^i(\Hh)$. 
\end{proof}

Consider $e:S\to G$. Recall from Lemma \ref{lem:res_functorial}
(with $\phi=[a]$, $D_1=D_2=e(S)$) that there is an $[a]$-linear operation
on the residue sequence 
\[e_!e^!\Log{n}\to \Log{n}\to Rj_*j^*\Log{n}\]
compatible with the operation on $\Log{n}$.

\begin{prop}\label{prop:decompose_for_e}
\begin{enumerate}
\item We have
\[ R\pi_!Re_{!}e^!\Log{n}=e^!\Log{n}=\bigoplus_{i=0}^n\Sym^n\Hh(-d)[-2d]\]
and $[a]$ operates on the $i$-th summand by multiplication by $a^i$.
\item The  object
$R\pi_{!}Rj_{*}j^*\Log{n}(-d)$ has a finite decomposition into generalized eigenspaces for the operation of $[a]$ with $a\in \Z$. The eigenvalues are
$a^i$ for $1\leq i\leq n+\kd(G)$. 
\item For $a\neq \pm 1$, the generalized $[a]$-eigenspace of
$R\pi_{!}Rj_{*}j^*\Log{n}(-d)$  for the eigenvalue $a^i$ is
given by  $\Sym^i\Hh(-d)[-2d+1]$
 via the residue map. It is actually an eigenspace, i.e., $[a]$ operates
 by multiplication by $a^i$.
\end{enumerate}
The decomposition is independent of the choice of $a$.
\end{prop}
\begin{proof}
The formula for $e^!\Log{n}$ is given in 
Lemma \ref{lem:log_in_e}. The operation of $[a]$ is the same as on the
associated gradeds of $\Log{n}$. By Theorem \ref{thm:functoriality} and
Theorem~\ref{log-funct}, respectively, it has the shape claimed in the Proposition.

%By definition in Lemma \ref{log-funct} the $[a]$-linear operation on $\pi^*\Sym^i\Hh$ is induced from $[a]$ on $G^{i+1}$ seen as a $G$-scheme
%via the projection to the first component. We pull back via $e:S\to G$.
%The operation on $e^*\pi^*\Sym^i\Hh=\Sym^i\Hh$ is induced from $[a]$-multiplication on $S\times G^i$. By \cite[Theorem 3.3]{AHP} (see also the review in Section \ref{sec:motive_cgs}) it is given by multiplication by $a^i$.

Consider the triangle on $G$
\[e_!e^!\Log{n}\to \Log{n}\to Rj_*j^*\Log{n}\ .\]
It induces an exact triangle on $S$
\[ \bigoplus_{i=0}^n\Sym^i\Hh(-d)[-2d]\to R\pi_!\Log{n}\to R\pi_!Rj_*j^*\Log{n} .\]
By the first assertion and Proposition \ref{eigenweights}, the first 
two objects have a finite decomposition into generalized $[a]$-eigenvalues
with eigenvalues as stated. 
Hence by Proposition \ref{triangulated-eigenspaces} the object on the right also
has a finite decomposition into generalized eigenspaces. We pass
to the generalized eigenspace for the eigenvalue $a^1$ and get 
\[ \Hh(-d)[-2d]\to 0\to ?\]
This proves the last 
assertion.

The decompositions are independent of $a$ by Lemma \ref{simultan} because the different $[a]$ commute and the assertion is true for $\Sym^i\Hh$.
\end{proof}

As before
let $\iota_D:D\to G$ be the inclusion of a closed subscheme which is \'etale over $S$ and contained in $G[N]$ for some $N$. Let $a\in\Z$ such
that $[a]^{-1}D\subset D$.
Recall from Lemma \ref{lem:res_functorial}
(with $\phi=[a]$, $D_1=D_2=D$) that there is an $[a]$-linear operation
on the residue sequence 
\[ R\pi_{D!}\pi_D^!\Log{n}\to \Log{n}\to Rj_{D*}j_D^*\Log{n}\]
compatible with the operation on $\Log{n}$.

\begin{prop}\label{prop:decompose_for_N}
Let $\iota_D:D\to G$ be as before. Let $a\in\Z$ such that $D\subset [a]^{-1}D$. 

Then the object $R\pi_!Rj_{D*}j_D^*\Log{n}(d)$ has a finite decomposition into
generalized eigenspaces for the operation of $[a]$. 

For $a\neq \pm 1, 0$, the generalized $[a]$-eigenspace for the eigenvalue $a^0$ sits in a
distinguished triangle
\[ \left(R\pi_!Rj_{D*}j_D^*\Log{n}(-d)\right)^{a^0}\to R\pi_{D!}\Q[-2d+1]\to \Q[-2d+1]\]
via the residue map.

If $a,b\in\Z$ are integers with $D\subset [a]^{-1}D,[b^{-}]D\subset [ab]^{-1}D$, then they the decompositions with respect to $a$ and $b$ agree.
\end{prop}
\begin{rem}
The assumptions on $a$ are satisfied if $D\subset G[N]$ and $a\equiv 1 \mod N$.
\end{rem}
\begin{proof}
The arguments are the same as in the proof of Proposition \ref{prop:decompose_for_e}. It remains to compute explicitly for the eigenvalue $a^0$.
We apply $R\pi_!$ to the localization triangle
and pass to the generalized $[a]$-eigenspace for the eigenvalue $a^0$.
The eigenspace for $R\pi_!\Log{n}$ was computed in Proposition \ref{eigenweights} (3). The eigenspace for 
\[ R\pi_!\iota_{D!}\iota_D^!\Log{n}(d)=R\pi_{D!}\iota_D^*\Log{n}[-2d]=\pi_{D!}\bigoplus_{i=0}^n\Sym^i\pi_D^*\Hh [-2d]\]
is given by the summand for $i=0$.

Under the compatibility assumption on $a$ and $b$, it is easy to check 
along the lines of the proof of Lemma \ref{lem:res_functorial} that
the induced operations commute. Hence the decompositions agree by Lemma \ref{simultan}.
\end{proof}

\begin{proof}[Second Proof of Theorem \ref{polylog_all} and Theorem \ref{thm:polylog_divisor}.]
We want to construct an element in $\Ext^{2d-1}_S(\Hh,\pi_{!}j_*\Log{n}|_U(d))$
Choose $a\in\Z$, $a\neq \pm 1,0$.
We define 
\[\pol{n}\in \Ext^{2d-1}_S(\Hh,\pi_{!}j_*\Log{n}|_U(d))\]
be the unique preimage of $\id\in\Hom(\Hh,\bigoplus_{i=0}^n\Sym^n\Hh)$
under the residue map of Definition \ref{defn:res_motivic}
such that $\pol{n}$ 
maps to the generalized $[a]$-eigenspace of $\pi_!j_*\Log{n}$ with eigenvalue $a^1$.

By construction it is compatible under restriction and with the
realization functors. By uniqueness, it is also functorial with respect
to group homomorphisms $\phi:G_1\to G_2$. In particular, $\pol{n}$ is
independent of the choice of $a$.

Now let $\alpha\in \Q[D]^0$. We choose $a\in\Z$ with $a\neq \pm 1,0$ such
that $[a]^{-1}D\subset D$, e.g., $a\equiv 1\mod N$ with $D\subset G[N]$.
We define
\[\pol{n}_\alpha\in\Ext^{2d-1}(\bQ,\pi_!j_{D*}\Log{n}(d))\]
as be the unique preimage of $\alpha$ under the residue map of Definition \ref{defn:res_motivic} which maps to the generalized $[a]$-eigenspace of $\pi_!j_{N!}\Log{n}$ for the eigenvalue $a^0$.  
By construction, it is compatible under restriction and with realization
functors. By uniqueness, it is also functorial with respect to
group homomorphisms $\phi:G_1\to G_2$ such that $\phi^{-1}D_2\subset D_1$.
In particular, it is independent of the choice of $a$.
\end{proof}

\section{Comparison with other definitions of the polylog}\label{sec:with_*}
We work in the sheaf theoretic and in the motivic setting in parallel.

In order to relate our constructions to the existing literature, we
also need a version of polylog with respect to $R\pi_*$.

\subsection{Comparing $R\pi_!$ and $R\pi_*$}
Recall that there is always a natural map of functors $R\pi_!\to R\pi_*$.

If $D\subset G$ is finite \'etale over $S$, then there is a commutative diagram
\begin{equation}\label{eq:comp-def}
\xymatrix{
R\pi_!Rj_{D*}j_D^{*}\cLog(d))[2d]\ar[rr]^\comp\ar[dr]&&
R\pi_*Rj_{D*}j_D^{*}\cLog(d)[2d]\ar[dl]\\
&\pi_{D*}\iota^*_D\cLog)[1].
}
\end{equation}
Let $D\subset G$ be finite \'etale over $S$ and contained in $G[N]$ for
some $N$. 
By applying $R\pi_*$ instead of $R\pi_!$, we obtain another variant of
the residue triangle:
\[  R\pi_*\iota_{D!}\iota_D^!\Log{n}(d)[2d]\to R\pi_*\Log{n}(d)[2d]\to R\pi_*Rj_{D*}j_D^*\Log{n}(d)[2d] .\]
Again under the identification of 
 Definition
\ref{defn:logn} and Definition \ref{defn:logn_motivic} and because $\iota_D$ is
proper, we have
\[ R\pi_*\iota_{D!}\iota_D^!\Log{n}(d)[2d]= R\pi_{D*}\bigoplus_{i=0}^n\pi_{D!}\Sym^i\pi_D^*\Hh.\]
Hence the connecting morphism induces by adjunction
another map, again called residue map,
\[  \Ext_{U_D}^{2d-1}(\sF,j_D^*\Log{n}(d))
\to \Hom_S(\sF,\bigoplus_{i=0}^n\pi_{D*}\Sym^i\pi_D^*\Hh).\]

\begin{lemma}\label{prop:ext-comp2} 
Let $\sF$ be an object of $\D(S)$. 
There is an exact sequence
\[ \Ext^{2d-1}_{U_D}(j_D^*\sF,j_D^{*}\cLog(d))\xrightarrow{\res}
\Hom_S(\sF, \pi_{D*}\iota_D^*\cLog)\to \Hom_S(\sF,\bQ).
\]
In the sheaf theoretic setting, let
$\sF$ be a sheaf on $S$. Then the residue map is injective.
\end{lemma}
\begin{proof}Same argument as for $R\pi_!$, see Lemma \ref{prop:ext-comp}.
\end{proof}

\subsection{Polylog with $R\pi_*$}
The map $\comp$ from \eqref{eq:comp-def} induces maps
\begin{equation}
\label{eq:pol-comp} \xymatrix{\Ext_S^{2d-1}(\Hh,R\pi_!Rj_*j^*\Log{n}(d))\ar[r]\ar[dr]&
\Ext_S^{2d-1}(\Hh,R\pi_*Rj_*j^*\Log{n}(d))\\
&\Ext_{U}^{2d-1}(\pi_U^*\Hh,j^*\Log{n}(d))\ar[u]_=}
\end{equation}
and similarly
\begin{equation}\label{eq:pol-alpha-comp} 
\Ext_S^{2d-1}(\bQ,R\pi_!Rj_{D*}j_D^*\Log{n}(d))\to
\Ext_{U_D}^{2d-1}(\bQ,j_D^*\Log{n}(d)).
\end{equation}
We define the polylog with respect to $R\pi_*$ as the image of the polylog under these maps.
\begin{defn}\label{defn:pol_with_*}
We denote by
\[
\polzwei{n}\in Ext_{U}^{2d-1}(\pi_U^*\Hh,j^*\Log{n}(d))
\]
the image of $\pol{n}$ under the map \eqref{eq:pol-comp} and for $\alpha\in\Q[D]^0$,
we denote by 
\[
\polzwei{n}_\alpha\in \Ext_{U_D}^{2d-1}(\bQ,j_D^*\Log{n}(d))
\]
the image of $\pol{n}_D$ under the map \eqref{eq:pol-alpha-comp}. 
\end{defn}
These classes have the advantage of having an interpretation on $U$ and $U_D$,
respectively. They have the disadvantage of having a more restrictive functoriality.

\begin{prop}
\begin{enumerate}
\item $\polzwei{n}$ and $\polzwei{n}_\alpha$ are compatible under the transition maps
$\Log{n}\to \Log{n-1}$. We write $\cpolzwei\in Ext_{U}^{2d-1}(\pi_U^*\Hh,j^*\cLog(d))$ and 
$\cpolzwei_\alpha\in \Ext_{U_D}^{2d-1}(\bQ,j_D^*\cLog(d))$ for the resulting classes.
\item $\polzwei{n}$ and $\polzwei{n}_\alpha$ are contravariantly functorial in the base scheme $S$.
\item The image of $\polzwei{n}$ under the residue map is given by
the natural inclusion of $\Hh$ into $\bigoplus_{n=0}^n\Sym^i\Hh$.
\item The image of $\polzwei{n}_\alpha$ under the residue map is given
by $\alpha$.
\item
Let $\phi:G_1\to G_2$ be a proper morphism of $S$-group schemes.
\begin{enumerate}
\item The diagram
\[\begin{CD}
     \phi_*\Hh_{G_1}@>\pol{n}_{G_1}>> \phi_*j_1^*\Log{n}_{G_1}(d_1)[2d_1-1]\\
     @V\varphi  VV @VV\varphi_\# V\\
 \Hh_{G_2}@>\pol{n}_{G_2}>> j_2^*\Log{n}_{G_2}(d_2)[2d_2-1]\\
\end{CD}\]
commutes.
\item
 the class $\polzwei{n}\alpha$ is mapped to
$\polzwei{n}_{\phi_!\alpha}$ under
\[ \phi_\#:\Ext^{2d_1-1}_{G_1}(\bQ,j_D^*\Log{n}_{G_1}(d))\to \Ext_{G_2}^{2d_2-1}(\bQ,j_{\phi(D)}^*\Log{n}_{G_2}(d)).\]
\end{enumerate}
\end{enumerate}
\end{prop}
\begin{proof}The argument as the same as in the proof of Theorem \ref{polylog_all}. The main ingredient is the functoriality of $\Log{n}$ in Theorem \ref{thm:functoriality}.
\end{proof}
Functoriality is of particular interest in the case where $\phi$ is an isogeny, e.g., multiplication
by $N$ with $N$ invertible on $S$. 
\begin{rem}
It is not clear in general if $\polzwei{n}$ and $\polzwei{n}_\alpha$ are uniquely determined by their residues. 
In a more special geometric situation, which covers the cases in the existing literature, uniqueness is at least true in the sheaf theoretic setting.
\end{rem}

\begin{prop}\label{prop:compare!*}
In the sheaf theoretic setting,
the map 
\[ \comp: \Ext^{2d-1}_{S}(\Hh,R\pi_!Rj_{D*}j_D^{*}\cLog(d))\to
\Ext^{2d-1}_{G\setminus D}(\Hh,\cLog(d))\]
is an isomorphism, if either
\begin{enumerate}
\item $G$ is an abelian scheme,
\item  
$G$ is an extension of an abelian scheme $A/S$ of dimension $g$  by a torus $T/S$ of dimension $r$,
and the considered sheaf theory admits weights.
\end{enumerate}
In these cases $\cpolzwei$ is uniquely determined by its compatibility under
the restriction maps or by functoriality for some $a\in\Z$, $a\neq 0,\pm 1$.
\end{prop}
%Recall that a group scheme is semi-abelian \guido{das ist nicht die richtige Definition} 
%if it is the extension 
%of an abelian scheme $A/S$ of dimension $g$ by a torus $T/S$
%of dimension $r$. 
%\[ 
%0\to T\to G\to A\to 0.
%\]
Note that in the second case $\Hh$ is a lisse of rank $h=2g+r$. 
\begin{proof}
If $G$ is an abelian scheme, 
the map $\comp$ is just the natural adjunction, hence an isomorphism and there is nothing to show.

Now let $G$ be an extension of $A/S$ by $T/S$ as in the statement.
Let $h:=\dim_\Q\sH=2g+r$, 
then by Theorem \ref{thm:vanishing} one has
\[
\Ext^{j}_{G}(\pi^*\Hh,\cLog(d))=\Ext^{j}_S(\Hh,R\pi_*\cLog(d))\isom\Ext^{j-h}_S(\Hh,\bQ).
\]
The weights of  $\sF=\sH$  are $\le -1$ and the
$\Ext$-groups $\Ext^{j-h}_S(\sF,\bQ)$ vanish. 
Then the localization sequence gives rise, with the arguments
from Proposition \ref{prop:ext-comp}, to an isomorphism
\[
\Ext^{2d-1}_{G\setminus D}(\pi^*\sH,\cLog(d))\isom
\Hom_S(\sH, \pi_{D*}\iota_D^*\cLog)
\]
because $\Hom_S(\sH,\bQ)=0$. Together with \ref{prop:ext-comp}
this shows that $\comp $ is an isomorphism.
\end{proof}

\subsection{Special cases}\label{sec:special_cases}
We review the existing literature and how the present paper
fits. In all cases, it is $\polzwei{n}$ and $\polzwei{n}_\alpha$ defined in Definition
\ref{defn:pol_with_*} that appears. Recall that for abelian schemes one has $\pol{n}=\polzwei{n}$. By Proposition \ref{prop:compare!*}, the class $\polzwei{n}$
is not identical, but has the same information as $\pol{n}$, at least in the sheaf theoretic setting.
\begin{enumerate}
\item If $G=\Gm$, then we are in the situation
of the classical polylog on the projective line minus three points. Its sheaf theoretic construction 
by Deligne in \cite{Del} was the starting point of the whole field.
The motivic construction over $S=\Z$ (that is enough by functoriality)
is due to Beilinson and Deligne. Full details can be found in \cite{HuWi} by Huber and Wildeshaus.
We are going to explain this case in more detail below.
\item If $G=E$ is an elliptic curve, 
it agrees with the sheaf theoretic polylog for elliptic curves
as defined by Beilinson and Levin \cite{BeLe}. They also constructed the motivic elliptic polylog. Their treatment served as the role model for all later definitions of the polylogarithm.  
\item If $G=A$ is abelian and $S$ is regular, the motivic polylog constructed
in the present paper agrees with the one constructed by the second author
in \cite{Ki}.  In this paper the decomposition under the $[a]$-operation, as used by Beilinson and Levin, was amplified and made into a flexible tool, which motivated the approach in the present paper. 
\item
If the considered sheaf theory admits weights and $G$ is
an extension of an abelian scheme by a torus, then the polylogarithm class $\polzwei{n}$
of Definition~\ref{defn:pol_with_*}
\begin{equation}
\cpolzwei\in \Ext^{2d-1}_{G\setminus\{e\}}({\pi}^*\sH,\cLog(d))
\end{equation}
agrees with the polylogarithm
defined by Wildeshaus in \cite[page 161]{Wi}. In particular, we achieve the construction of the motivic classes inducing his sheaf theoretic polylogarithm.
\end{enumerate}

\subsection{Classical polylog}\label{sec:classical}
As the case $G=\Gm$ is of particular interest, and our approach is a considerable technical simplification of the existing motivic construction in
\cite{HuWi}, we spell out the details.
It suffices to consider $S=\Spec \Z$.
We work in the motivic and sheaf theoretic setting in parallel.

\begin{lemma}For $G=\Gm$ we have 
\[ M_1(G)=\bQ(1)[1],\ \Hh_G=\bQ(1), \text{\ and\ }\Sym^k\Hh=\bQ(k).\]
Moreover,
\[ R\pi_!\bQ(k)=\bQ(k)\oplus\bQ(k+1)[1].\]
with the splitting induced by the unit section.
\end{lemma}
\begin{proof} The first statement is a classical computation of Voevodsky:
$\Z(1)[1]$ is represented by the sheaf $\mathcal{O}^*=\ul{\Gm}$.
\cite[Theorem 3.4.2]{voe}. All the others follow.
\end{proof}
This means that $\Log{n}$ is an iterated extension of Tate motives/sheaves on
$\Gm$.

\begin{defn}Let $S$ be finite dimensional and noetherian. The {\em triangulated category $D_{\MT}(S)$
of mixed Tate motives} on $S$ is defined as the full triangulated subcategory of $\DA(S)$ generated by $\bQ(k)$ for $k\in\Z$.

\end{defn}
Note that this category is closed under tensor products and duality. 

We say that Tate motives on $S$ satisfy the Beilinson-Soul\'e vanishing conjectures if
\[ \Hom_{\DA(S)}(\bQ(i),\bQ(j)[N])=0\]
for all $N<0$. This implies
 the existence of a t-structure on $D_{\MT}(\Spec\Z)$ such
that the Betti- or $\ell$-adic realizations are t-exact and conservative.

\begin{defn}Let $\MT(S)$ be the {\em abelian category of mixed Tate motives}
on $S$ be defined as the heart of the motivic t-structure on $D_{\MT}(\Spec\Z)$.
\end{defn}

\begin{lemma}Tate motives on $\Spec\Z$, $\Gm$ and $U$ satisfy
the Beilinson-Soul\'e vanishing conjectures.
\end{lemma}
\begin{proof}
Borel's computation of higher algebraic $K$-theory of $\Z$ implies the
case of $S=\Spec\Z$.

For $S=\Gm$ we consider
\begin{align*}
\Hom_{\Gm}(\bQ(i),\bQ(j)[N])&=\Hom_{\Spec \Z}(R\pi_!\pi^!\bQ(i),\bQ(j)[N])\\
 &=\Hom_{\Spec \Z}(\bQ(i)\oplus \bQ(i+1)[1],\bQ(j)[N])\\
 &=\Hom_{\Spec \Z}(\bQ(i),\bQ(j)[N])\oplus\Hom_{\Spec \Z}(\bQ(i+1),\bQ(j)[N-1]).
 \end{align*}
 Both summands vanish for $N<0$.

For $S=U$  consider the localizing triangle
\[ R\pi_{U!}\pi_U^!\bQ(i)\to R\pi_!\pi^!\bQ(i)\to e_*e^*\bQ(i+1)[2]\] 
and the long exact sequence for $\Hom_{\Spec\Z}(\cdot,\bQ(j)[N])$ to get
the same vanishing.
\end{proof} 

\begin{cor}The motives $\Log{n}$ and $j^*\Log{n}$ are objects of $\MT(\Gm)$
and $\MT(U)$, respectively.

The motives $R\pi_!\Log{n}$ and $R\pi_!j_{*}j^*\Log{n}$ are objects of the triangulated category of  mixed Tate
motives on  $\Spec\Z$. 
\end{cor}
\begin{proof} Immediate from the triangle
\[ \bQ(n)\to \Log{n}\to\Log{n-1}\]
the computation of $R\pi_!\pi^!\bQ$.
\end{proof}

Hence the spectral sequence computation of Section \ref{proof:vanishing} and its conclusion in Theorem \ref{thm:vanishing} are also true in the motivic setting.
Note that the argument simplifies considerably in this special case, see
\cite[Appendix A]{HuKiinventiones} for the cohomological case. The homological
case agrees with this up to a shift because $\bQ(i)^\vee=\bQ(-i)$. 

\begin{cor} The localization sequence with respect to the unit section
$e:\Spec\Z\to \Gm$ induces a long exact sequence
\[ \bQ(-1)\to R^1\pi_!j_*j^*\Log{n}(1)\to \bigoplus_{k=0}^n\bQ(k)\to 0\]
of mixed Tate motives.
\end{cor}

Moreover, the proof of Proposition \ref{prop:compare!*} also applies
in the motivic setting because the theory of mixed Tate motives has weights.

\begin{defn}Let $\pol{n}\in\Ext^1_{\Spec \Z}(\bQ(1),R\pi_!j_*j^*\Log{n})$ be
the unique element with residue the natural inclusion
$\bQ(1)\to\bigoplus_{k=0}^n\bQ(k)$.

Let $\polzwei{n}\in\Ext^1_{\Gm}(\bQ(1),j^*\Log{n})$ be the unique element
with residue the natural inclusion $\bQ(1)\to\bigoplus_{k=0}^n\bQ(k)$.
\end{defn}

\begin{rem}
\begin{enumerate}
\item 
The analogous discussion can also be carried out for
$\pol{n}_\alpha$. It involves Artin-Tate motives because
$R\pi_{D!}\pi_{D}^!\bQ$ is Artin-Tate. Borel's result on
motivic cohomology is still available. 
We omit the precise formulation.
\item The same arguments are also valid for all tori over
a base $S$ where Tate motives satisfy the Beilinson-Soul\'e vanishing
conjectures.
\end{enumerate}
\end{rem}

\section{Proof of the vanishing theorem}
\label{proof:vanishing}
\subsection{Proof of Theorem \ref{thm:vanishing}}
We work in the sheaf theoretic setting.

Before we give the proof we start with some general remarks concerning $R\pi_!\bQ$ and the definition of $\Log{1}$. First note
that the group multiplication $\mu:G\times_S G\to  G$ induces 
a product 
\[
\mu:R^{i}\pi_!\bQ(d)\otimes R^{j}\pi_!\bQ(d)\to R^{i+j-2d}\pi_!\bQ(d)
\]
and the diagonal $\Delta:G\to G\times_S G$ a coproduct
\[
\Delta: R^{i}\pi_!\bQ(d)\to \bigoplus_j R^{j}\pi_!\bQ(d)\otimes 
R^{2d+i-j}\pi_!\bQ(d).
\]
In particular, $\bigoplus_i R^{i}\pi_!\bQ(d)$ is a Hopf algebra
and a direct computation shows that $R^{2d-1}\pi_!\bQ(d)=\sH$
are the primitive elements. As usual we get an isomorphism 
\begin{equation*}
R^{i}\pi_!\bQ(d)\isom  \bigwedge^{2d-i}\sH.
\end{equation*}
Recall that we have given a description of
$\Log{1}$ in terms of the comultiplication in Lemma \ref{lem:compute_once}.

We want to compute $R\pi_!\cLog$ by using the spectral sequence
arising from the unipotent filtration on $\cLog$. For
this we need to identify the connecting homomorphisms.
\begin{lemma}\label{lemma:boundary}
The connecting homomorphism
\[
R^i\pi_!\bQ\to R^{i+1}\pi_!\pi^*\sH\isom R^{i+1}\pi_!\bQ\otimes \sH 
\]
of the long exact cohomology sequence of
\[
0\to \pi^*\sH\to \Log{1}\to \bQ\to 0
\] 
is given (up to sign) by the composition of the comultiplication
\[
\Delta: R^{i}\pi_!\bQ\to \bigoplus_j R^{j}\pi_!\bQ\otimes 
R^{2d+i-j}\pi_!\bQ(d)
\]
with the projection onto $R^{i+1}\pi_!\bQ\otimes R^{2d-1}\pi_!\bQ(d)$.
\end{lemma}
\begin{proof}
This is completely formal. The comultiplication is obtained 
by applying $R(\pi\times \pi)_!$ to 
\[
\Delta_!\Delta^{!}(\pi\times \pi)^{!}\bQ\to (\pi\times \pi)^{!}\bQ.
\]
We factor $R(\pi\times \pi)_!=R(\id\times\pi)_!\circ R(\pi\times \id)_!$ and get that the comultiplication is given by applying $R\pi_!$ to the map
$\pi^{!}\bQ\to \pi^{!}R\pi_!\pi^{!}\bQ$. On the other hand, the connecting homomorphism is obtained by applying $R\pi_!$ 
to the composition $\pi^{!}\bQ\to \pi^{!}R\pi_!\pi^{!}\bQ\to
\pi^{!}\sH[1]$ from \eqref{eq:comp}, which by the above lemma describes the extension $\Log{1}$.
 \end{proof}
 To compute the higher direct images  
of $\cLog^{(n)}$ we need the
exact Koszul complex (see \cite[4.3.1.7]{Illusie})
\begin{equation}\label{Koszulcomplex}
        0\to 
        \bigwedge^{m}\sH \xrightarrow{d^{0}_m}\ldots\xrightarrow{d^{i-1}_m}
        \bigwedge^{m-i}\sH\otimes \Sym^i\sH\xrightarrow{d^{i}_m}\ldots
        \xrightarrow{d^{m-1}_m}\Sym^{m}\sH\to 0.
        \end{equation}
Recall that the differentials 
$d^{i}_m:       \bigwedge^{m-i}\sH\otimes\Sym^{i}\sH\to
\bigwedge^{m-i-1}\sH\otimes\Sym^{i+1}\sH$
are induced by the comultiplication $\bigwedge^{m-i}\sH\to \bigwedge^{m-i-1}\sH \otimes \sH$ of the exterior algebra composed with the multiplication
of the symmetric algebra.

\begin{prop}\label{Log-cohomology}
The spectral sequence associated to the filtration of $\Log{n}$ by unipotence length 
$$
E_1^{p,q}=R^{p+q}\pi_!\pi^*\Sym^{p}\sH(d)\Rightarrow R^{p+q}\pi_!\Log{n}(d).
$$
has $E_1^{p,q}\isom\bigwedge^{2d-p-q}\sH\otimes\Sym^{p}\sH$ for
$0\leq p\leq n$ and $p+q\geq 0$
and $E_1$-differential
given by the Koszul differential. It degenerates at $E_2$ with
\[
R^i\pi_!\cLog^{(n)}(d)\isom\begin{cases}
\bQ&i=2d\\
\coker\; d^{n-1}_{2d-i+n} & 0<i< 2d\\
R^0\pi_!\bQ(d)\otimes \Sym^n\sH & i=0.
\end{cases}
\]
where $d^{n-1}_{2d-i+n}:\bigwedge^{2d-i+1}\sH\otimes \Sym^{n-1}\sH\to 
\bigwedge^{2d-i}\sH\otimes \Sym^n\sH$ is the Koszul differential from 
\eqref{Koszulcomplex}.
\end{prop}
\begin{proof}  
The sheaf $\Log{n}$
has a filtration $F^\cdot\Log{n}$ such that the 
associated graded pieces are the $\pi^*\Sym^k\sH$ for 
$0\le k\le n$. We consider the associated spectral sequence
$$
E_1^{p,q}=R^{p+q}\pi_!\pi^*\Sym^{p}\sH(d)\Rightarrow R^{p+q}\pi_!\Log{n}(d).
$$
If we identify $R^{p+q}\pi_!\bQ(d)\isom \bigwedge^{2d-p-q}\sH$ we get 
\[
E_1^{\cdot,q}:0\to \bigwedge^{2d-q}\sH\xrightarrow{d_1^{0,q}}
\bigwedge^{2d-q-1}\sH\otimes \sH\xrightarrow{d_1^{1,q}} \ldots
\xrightarrow{d_1^{n-1,q}}
\bigwedge^{2d-q-n}\sH\otimes \Sym^n\sH,
\]
where the first term is $E_1^{0,q}$ etc.
We assume by induction on $n$ that the differentials $d_1^{p,q}$  in the spectral sequence for $\Log{n-1}$ are
the Koszul differentials $d^{p}_{2d-q}$ (up to sign). 
The case $n=1$ is Lemma \ref{lemma:boundary}.
Then the $d_1^{p,q}$ for $p\le n-2$ in the
spectral sequence for $\Log{n}$ coincide also with
the Koszul differentials $d^{p}_{2d-q}$ (up to sign). 
We claim that
this is also true for $d_1^{n-1,q}$.
The differentials $d_1^{n-1,q}$ in the spectral sequence are the connecting 
homomorphisms for $R^{n-1+q}\pi_!$ of the short exact sequence
\[
0\to \pi^*\Sym^{n}\sH\to F^{n-1}\Log{n}/F^{n+1}\Log{n} \to
\pi^*\Sym^{n-1}\sH\to 0.
\]
By construction  of $\Log{n}$  this short exact sequence is
isomorphic to the push-out of 
\[
0\to  \pi^*\Sym^{n-1}\sH\otimes \pi^*\sH\to  \pi^*\Sym^{n}\sH\otimes
\Log{1}\to  \pi^*\Sym^{n-1}\sH\to 0
\]
by the multiplication map $\pi^*\Sym^{n-1}\sH\otimes \pi^*\sH\to
\pi^*\Sym^{n}\sH$. In particular, the connecting homomorphisms
are the ones of $\Log{1}$ tensored with $\pi^*\Sym^{n-1}\sH$.
If one unravels the definitions one gets that
the $d_1^{n-1,q}$ are the Koszul differentials.

It follows that $E_1^{\cdot,q}$ is
the truncated Koszul complex and hence 
the only non-zero $E_2$-terms are $E_2^{0,2d}=\bQ$, 
$E_2^{n,q}=\coker\; d_1^{n-1,q}$ for $-n<q<2d-n$ and 
$E_2^{n,-n}=R^0\pi_!\bQ(d)\otimes \Sym^n\sH$. For the higher direct 
images we get accordingly
\[
R^i\pi_!\cLog^{(n)}(d)=\begin{cases}
\bQ&i=2d\\
\coker\; d_1^{n-1,i-n}&0<i<2d\\
R^0\pi_!\bQ(d)\otimes\Sym^n\sH&i=0,
\end{cases}
\]
which is the desired result.
\end{proof}
As a corollary we get the statement of Theorem \ref{thm:vanishing}:
\begin{cor}\label{cor:vanishing}
One has 
\[
R^i\pi_!\cLog \isom \begin{cases}
\bQ(-d) & i=2d\\
0 & i\neq 2d.
\end{cases}
\]
\end{cor}
\begin{proof} From the computation
of $R^{2d}\pi_!\Log{n}$ it follows that the transition maps
$R^{2d}\pi_!\Log{n}\isom R^{2d}\pi_!\Log{n-1}$ are all isomorphisms. In particular, $R^{2d}\pi_!\cLog\isom \bQ(-d)$.

It remains to show that  $R^i\pi_!\cLog=0$ for $i\neq 2d$
and for this it is enough to show that  $R^i\pi_!\Log{n}\to R^i\pi_!\Log{n-1}$ is
the zero map. Consider the long exact cohomology sequence of 
$$
0\to \pi^*\Sym^n\sH\to \Log{n}\to \Log{n-1}\to 0.
$$
By the computation of $R^i\pi_!\Log{n}$ in
Proposition \ref{Log-cohomology} the map 
$$
R^i\pi_!\Sym^n\sH\isom \bigwedge^{2d-i}\sH\otimes \Sym^n\sH\to R^i\pi_!\Log{n}
$$
is surjective, hence $R^i\pi_!\Log{n}\to R^i\pi_!\Log{n-1}$ is
the zero map.
\end{proof}

We now turn to the case where $G$ is an extension
of an abelian scheme by a torus
and hence $\Hh$ locally constant.
We discuss the necessary modifications of this proof to
get the statement for the higher direct images $R^i\pi_*\cLog$.
First note that one has by  Poincar\'e duality a perfect pairing 
\[
R^{i}\pi_*\bQ\otimes R^{2d-i}\pi_!\bQ\to \bQ(-d),
\]
which shows  $\bigoplus_i R^{i}\pi_*\bQ\isom \bigwedge^{i}\sH^{\vee}$. The dual of the quasi-isomorphism in
\eqref{eq:coh-splitting} gives the decomposition  
\begin{equation}\label{eq:coh-decomposition}
R\pi_*\pi^{*}\bQ\isom \bQ\oplus \tau_{>0}R\pi_*\pi^{*}\bQ.
\end{equation}
To identify the extension class 
of $\Log{1}\in \Ext^{1}_G(\bQ,\pi^{*}\sH)$ consider the evaluation
map $\ev:\sH\otimes \sH^{\vee}\to \bQ$ and its dual 
\[
\ev^{\vee}:\bQ\to \sH^{\vee}\otimes \sH.
\]
Note further that by duality one has 
\[
R\pi_*\pi^{*}\sH\isom (R\pi_!\pi^{!}\sH^{\vee})^{\vee}\isom (R\pi_!\pi^{!}\bQ\otimes\sH^{\vee})^{\vee}\isom R\pi_*\pi^{*}\bQ\otimes\sH.
\]
\begin{lemma}
The class of $\Log{1}\in \Ext^{1}_S(\bQ,R\pi_*\pi^{*}\sH)=
\Hom_S(\bQ, R\pi_*\pi^{*}\sH[1])$ is given by the composition 
\[
\bQ\xrightarrow{\ev^{\vee}}\sH^{\vee}\otimes \sH\to 
\tau_{>0}R\pi_*\pi^{*}\bQ\otimes \sH[1]\xrightarrow{\eqref{eq:coh-decomposition}}R\pi_*\pi^{*}\bQ\otimes \sH[1]
\]
where the arrow in the middle is induced by the map 
$\sH^{\vee}=R^{1}\pi_*\pi^{*}\bQ\to \tau_{>0}R\pi_*\pi^{*}\bQ[1]$.
\end{lemma}
\begin{proof}
By definition the extension class
$\Log{1}\in \Hom_S(R\pi_!\pi^{!}\bQ,\sH[1] )$ is given by the map
in \eqref{eq:comp}, which induces 
\[
\bQ\to (\tau_{\le -1}R\pi_!\pi^{!}\bQ)^{\vee}\otimes \sH[1]\isom \tau_{>0}R^{1}\pi_*\pi^{*}\bQ\otimes \sH[1].
\]
If one unravels the definition of the map in \eqref{eq:comp} one gets the map in the lemma.
\end{proof}
Let $h:=\dim_\Q\sH$ be the dimension of the local system  $\sH$, then the pairing $\bigwedge^i\sH^{\vee}\otimes\bigwedge^{h-i}\sH^{\vee}\to
\bigwedge^h\sH^{\vee}$ induces an isomorphism
\[
R^i\pi_*\pi^{*}\bQ\isom \bigwedge^i\sH^\vee\isom \bigwedge^{h-i}\sH\otimes\bigwedge^h\sH^{\vee} .
\]
The computation of $R^i\pi_*\cLog^{(n)}(d)$ is exactly the same as before, once we have identified the connecting homomorphisms
\[
R^i\pi_*\bQ\to R^{i+1}\pi_*\pi^*\sH 
\]
of the extension 
\[
0\to \pi^{*}\sH\to \Log{1}\to \bQ\to 0.
\]
\begin{lemma} Using the above identification 
$R^i\pi_*\pi^{*}\bQ\isom\bigwedge^{h-i}\sH\otimes\bigwedge^h\sH^{\vee} $ the connecting homomorphism
\[
\bigwedge^{h-i}\sH\otimes\bigwedge^h\sH^{\vee} \to
\bigwedge^{h-i-1}\sH\otimes\sH\otimes \bigwedge^h\sH^{\vee} 
\]
is induced by the comultiplication in $\bigwedge^{\cdot}\sH$.
\end{lemma}
\begin{proof}
The connecting homomorphism is the map 
\[
R^i\pi_*\bQ\to R^i\pi_*\bQ\otimes R^{1}\pi_*\bQ\otimes \sH\to  R^{i+1}\pi_*\bQ\otimes\sH
\]
induced by the cup-product. If we make the identifications explicit, we get the desired formula.
\end{proof}
Exactly as in the proof of
Proposition \ref{Log-cohomology} we get
\[
R^i\pi_*\cLog^{(n)}\isom\begin{cases}
\bigwedge^h\sH^{\vee}&i=h\\
\bigwedge^h\sH^{\vee}\otimes \coker\; d^{n-1}_{2h-i+n} & 0<i< h\\
\bigwedge^h\sH^{\vee}\otimes \Sym^n\sH & i=0.
\end{cases}
\]
\begin{cor}\label{cor:vanishing_semiabelian}
Let $G$ be an extension of an abelian scheme by a torus. Then
\[
R^i\pi_*\cLog \isom \begin{cases}
\bigwedge^h\sH^{\vee}& i=h\\
0 & i\neq h.
\end{cases}
\]
\end{cor}
\begin{proof}
This follows by the same argument as in Corollary \ref{cor:vanishing}.
\end{proof}
\subsection{Proof of Theorem \ref{thm:vanishing-motivic}.}\label{sec:motivic_proof}
We turn to the motivic setting with either $S$ of characteristic $0$
or $G$ affine.
As always, $G/S$ is a smooth commutative group scheme with connected fibres.

\begin{lemma}\label{lem:precise_weights}
Let $S$ be a scheme of characteristic $0$ or $G$ affine.
Let $a\in\Z$, $a\neq 0,\pm 1$. Then the generalized $a^0$-eigenspace
for the operation of $[a]$ on  $R\pi_!\Log{n}$ is isomorphic to $\Q(-d)[-2d]$. 
The generalized $a^j$-eigenspace vanishes for
$j>n+\kd(G)$ and for $0<j<n$.
\end{lemma}
This is a refined version of the vanishing in Proposition \ref{eigenweights} (3). Its proof relies on much deeper input from the theory of motives. 
\begin{proof}
The computation of the generalized $a^0$ eigenspace was carried out
in Proposition \ref{eigenweights} (3).
The vanishing for $j>n$ follows simply by induction from the statement
for $R\pi_!\Sym^i\Hh$, see
Proposition \ref{eigenweights} (1).

We now turn to the essential part of the statement, with
$0<j< n$. We claim that the $a^j$-eigenspace vanishes.
By \cite[Lemma A.6]{AHP} it is enough to prove the statement
after base change to geometric points $\bar{s}:k\to S$. Moreover,
$\bar{s}^*$ is a tensor functor commuting with $R\pi_!$ and
$\bar{s}^*\Hh_{G/S}=\Hh_{G_{k}/k}$ by \cite[Proposition 2.7]{AHP}.
Hence we may assume without loss of generality that $S=\Spec k$ with
$k$ algebraically closed. We have been working in categories of \'etale
motives without transfers so far. In the case of a perfect ground field
$k$, the ''adding transfer'' functor is an equivalence of categories.
Hence we can argue in Voevodsky's orginal category of geometric motives
$\DM(k,\Q)$ from now on. 

We claim that the object $R\pi_!\Log{n}$ is contained in the subcategory
of abelian motives in the sense of \cite[Definition 1.1]{Wi14}. It
is the thick tensor triangulated subcategory of the category of
geometric motives generated by $\Q(r)$ for $r\in\Z$ and the Chow motives
of abelian varieties.
We can verify this by induction on $n$. We have computed
$R\pi_!\pi^*\Sym^i\Hh$ in the proof of Proposition \ref{eigenweights}.
Hence it suffices to establish the claim for $M_1(G)$. 
By \cite[Lemma 7.4.5]{AEH}, the motive $M_1(G)$ agrees with
the 1-motive of the semiabelian part $G^{sa}$ of $G$. In the semi-abelian
case, the sequence
\[ 1\to T\to G^{sa}\to A\]
with $T$ a torus and $A$ an abelian variety induces an exact triangle
$M_1(T)\to M_1(G)\to M_1(A)$. The torus $T$ is split because we have
assumed $k$ to be algebraically closed. Hence $M_1(T)=\Q(1)^r$ is in
the category of abelian motives. The motive $M_1(A)$ is a Chow motive
as a direct summand of the motive of $A$, hence also in the category of
abelian motives.

Let $\ell$ be a prime invertible
in $k$.
If $S$ is of characteristic $0$, we have verified the assumptions of \cite[Theorem 1.16]{Wi14}.
 By loc.cit. the $\ell$-adic realization $H^*R_\ell$ is conservative.
We have reduced the assertion to the same vanishing in the $\ell$-adic
setting. 
If $G$ is affine, then its motive is a mixed Tate motive. Again
the $\ell$-adic realization is conservative; this time via the conservative slice functors $c_n$ of \cite[Section 5]{HuKa}. 

Consider the computation of $R\pi_!\Log{n}_\ell$ in Proposition
\ref{Log-cohomology}. The proof shows that
the cohomology in degree $i<2d$ is
given by $E_2^{n,i-n}$ and a functorial quotient of 
$E_1^{n,i-n}=\bigwedge^{2d-i}\sH_\ell\otimes\Sym^{n}\sH_\ell$. The operation
of $[a]$ on this term is by multiplication by $a^{2d-i+n}$. Recall
that $0\leq i<2d$. There is no contribution
to the $a^j$-eigenspace for $0<j<n$.
\end{proof}

\begin{proof}[Proof of Theorem \ref{thm:vanishing-motivic}.]
We want to show that
\[ R\pi_!\cLog\to R\pi_!\Q\xrightarrow{\res_e} e^*\Q(-d)[-2d]\]
is an isomorphism in $\DA(S)$. We pass to generalized $a$-eigenspaces
for the operation of $a\in\Z$. It suffices to show:
\begin{enumerate}
\item The $a^0$-eigenspace of $R\pi_!\Log{n}$ is equal to $\Q(-d)[-2d]$
for all $n$.
\item For $i\geq 1$, the pro-object given by the generalized $a^i$-eigenspaces of $R\pi_!\Log{n}$ is isomorphic to $0$.
\end{enumerate}
The first claim was shown in Proposition \ref{eigenweights} (3).
The second claim is a consequence of Lemma \ref{lem:precise_weights}.
\end{proof}

\begin{rem}It is tempting to remove the characteristic $0$ hypothesis
from the result. It enters the argument
via the proof of \cite[Theorem 1.13]{Wi14}, where it is used that homological and numerical equivalence agree on abelian varieties. This is open in positive characteristic.
\end{rem}

\begin{appendix}

\section{Eigenspace decomposition}\label{appA}
The aim of this section is verify the existence of decomposition
into generalized eigenspaces in the setting of triangulated categories.

\begin{defn}Let $\Ah$ be a pseudo-abelian $\Q$-linear additive category. Let $X$ be
an object and $\phi:X\to X$ an endomorphism. We say that $X$ has
a {\em finite decomposition into generalized $\phi$-eigenspaces} if
there is a $\phi$-equivariant direct sum decomposition
\[ X=\bigoplus_{i=1}^n X_i\]
together with a sequence $\alpha_1,\dots, \alpha_n$ of pairwise distinct rational numbers (''eigenvalues'') and a sequence $m_1,\dots,m_n$ of positive integers such that
$(\phi-\alpha_i)^{m_i}$ vanishes on $X_i$. We call $X_i$ the {\em generalized eigenspace for the eigenvalue} $\alpha_i$.
\end{defn}

\begin{ex}
Let $\Ah$ be the category of finitely generated $\Q$-vector spaces. 
Every object has a finite decomposition into generalized $\phi$-eigenspaces
by putting $\phi$ in Jordan normal form.
\end{ex}

This is not the most general notion one could imagine, but it suffices
for our application. The condition is equivalent to the following:
We view $X$ as a
$\Q[T]$-module with $T$ operating via $\phi$.
The object $X$ has a finite decomposition into generalized $\phi$-eigenspaces
if and only if the operation of $\Q[T]$ factors via an Artin quotient
$\Q[T]/I$ with $I$ of the form $\prod_{i=1}^n(T-\alpha_i)^{m_i}$.
By the Chinese Remainder Theorem, we have a ring isomorphism
\[\Q[T]/I=\prod_{i=1}^n \Q[T]/(T-\alpha_i)^{m_i}\ .\]
The decomposition of $X$ is induced from the decomposition
of $1\in\Q[T]/I$ into projectors. In particular, the decomposition
is unique if it exists. 

\begin{lemma}Let $A\to B\to C$ be an exact sequence of (possibly infinite dimensional) 
$\Q$-vector spaces with operation of an endomorphism $\phi$.
Assume that $A$ and $C$ admit a finite
decomposition into generalized $\phi$-eigenspaces. Then so does $B$.
\end{lemma}
\begin{proof}By assumption $A$ is a $\Q[T]/I$-module and
$C$ a $\Q[T]/J$ module with $I$ and $J$ of the special shape above.
It is easy to check that $IJ$ annihilates $B$, hence $B$ also
admits a decomposition into generalized $\phi$-eigenspaces.
\end{proof}
\begin{prop}\label{triangulated-eigenspaces}Let $\Th$ be a $\Q$-linear pseudo-abelian triangulated category.
Let $A\to B\to C\to A[1]$ be an exact triangle and $\phi$ an endomorphism
of the triangle. Assume that $A$ and $C$ admit a finite decomposition into
generalized $\phi$-eigenspaces. Then so does $B$. Given $\alpha\in\Q$ the triangle
of generalized eigenspaces for the eigenvalue $\alpha$ is distinguished.
\end{prop}
\begin{proof}
Consider the exact sequence of $\Q$-vector spaces 
\[ \Hom_\Th(B,A)\to \Hom_\Th(B,B)\to\Hom_\Th(B,C)\ .\]
By functoriality, it has an operation of $\phi$. As $A$ and $C$ have
a decomposition, so have $\Hom_\Th(B,A)$ and $\Hom_\Th(B,C)$. By the
lemma this implies that $\Hom_\Th(B,B)$ has decomposition. Equivalently,
$\Hom_\Th(B,B)$ is annihilated by an ideal $I$ of the special form above.
In particular, this is the case for $\id_B$ and hence for $B$. This
means that $B$ is an $\Q[T]/I$-module, or equivalently that it
admits a finite decomposition into generalized $\phi$-eigenspaces.

The ideal $I$ can be chosen such that it annihilates all of $A$, $B$, $C$.
This means that $\Q[T]/I$ operates on the exact triangle. The
decomposition of $B$ is compatible with the exact triangle.
Summing the triangles for all $\alpha\in\Q$ we get back
the original triangle. Hence the indivual triangles for fixed $\alpha$
are distinguished.
\end{proof}

\begin{lemma}\label{simultan}
Let $\Ah$ be a pseudo-abelian $\Q$-linear additive category. Let
$X$ be an object and $\phi:X\to X$ and $\psi:X\to X$ commuting endomorphisms.
Assume that $X$ has a finite decomposition into generalized eigenspaces for
$\phi$ and $\psi$. Then there is a unique simultaneous decomposition. 
\end{lemma}
\begin{proof}The operation of $\phi$ and $\psi$ make $X$ into 
a $\Q[T,S]$-module. By assumption $X$ is annihilated by a polynomial
$P=\prod_{i=1}^n(T-\alpha_i)^{n_i}$ and also by a polynomial
$Q=\prod_{j=1}^m(S-\beta_j)^{m_i}$. Hence the operation factors
via the Artinian ring 
$\Q[T,S]/(P,Q)$. By the Chinese Remainder Theorem, we have a ring
isomorphism
\[ \Q[T,S]/(P,Q)=\prod_{i,j}\Q[T,S]/((T-\alpha_i)^{n_i},(S-\beta_j)^{m_j}).\]
The decomposition of $X$ is induced from the decomposition
of $1$ into projectors.
\end{proof}

\end{appendix}

\end{document}